\documentclass[11pt, a4paper]{article}

\usepackage[
  top=2.7cm,
  bottom=2.9cm,
  left=2.5cm,
  right=2.5cm
]{geometry}

\usepackage{amsmath}
\usepackage{amsthm}
\usepackage{amssymb}

\usepackage[utf8]{inputenc}
\usepackage{hyperref}
\hypersetup{
	unicode,
	pdfauthor={Karl Kunisch, Donato Vásquez-Varas},
	pdftitle={Structure Preserving Approximation of semiconcave Functions}
}

\usepackage[sort&compress,numbers,square]{natbib}
\DeclareSymbolFont{matha}{OML}{txmi}{m}{it}
\DeclareMathSymbol{\varv}{\mathord}{matha}{118}

\DeclareSymbolFont{matha}{OML}{txmi}{m}{it}
\DeclareMathSymbol{\varv}{\mathord}{matha}{118}

\usepackage{todonotes}
\usepackage{caption}
\usepackage{subcaption}
\usepackage{makecell}
\usepackage{booktabs}
\usepackage{tabularray}
\UseTblrLibrary{booktabs}
\usepackage{comment}

\usepackage{amsthm}
\usepackage{algorithm}
\usepackage[noend]{algpseudocode}
\usepackage{todonotes}
\usepackage{caption}
\usepackage{subcaption}
\usepackage{makecell}
\usepackage{booktabs}
\usepackage{tabularray}
\usepackage{comment}

\usepackage{amsthm}
\usepackage{algorithm}
\usepackage[noend]{algpseudocode}

\providecommand{\norm}[1]{\left\lVert#1\right\rVert}
\newcommand{\dive}{\mbox{div}}

\newcommand{\tcr}{\color{red}}

\newcommand{\be}{\begin{equation}}
\newcommand{\ee}{\end{equation}}
\newcommand{\ba}{\begin{eqnarray}}
\newcommand{\ea}{\end{eqnarray}}
\newcommand{\beq}{\begin{equation}}
\newcommand{\eeq}{\end{equation}}

\renewcommand{\leq}{\leqslant}
\renewcommand{\le}{\leqslant}
\renewcommand{\geq}{\geqslant}
\renewcommand{\ge}{\geqslant}

\def \R {\mathbb{R}}

\def \N {\mathbb{N}}
\def \Om {\Omega}

\def\beq{\begin{equation}}
\def\eeq{\end{equation}}
\def\ecart{\noalign{\medskip}}
\def\ba{\begin{array}}
\def\ea{\end{array}}

\usepackage{xpatch}

\makeatletter
\xpatchcmd{\smash}{\ifmmode}{\ifmmode\count@=\fam}{}{}
\xpatchcmd{\mathsm@sh}{\m@th}{\m@th\fam=\count@}{}{}
\makeatother

\DeclareMathOperator*{\argmin}{\smash[b]{arg\,min}}

\usepackage[nameinlink]{cleveref}
\numberwithin{equation}{section}
\newtheorem{theo}{Theorem}[section]
\newtheorem{prop}{Proposition}[section]
\newtheorem{cor}{Corollary} [section]
\newtheorem{lemma}{Lemma}[section]
\newtheorem{defi}{Definition}[section]
\newtheorem{hypo}{Hypothesis}[section]

\newtheorem{rem}[defi]{Remark}

\usepackage{graphicx, color}
\graphicspath{{fig/}}

\usepackage{algorithm, algpseudocode} 
\usepackage{mathrsfs} 

\title{Structure Preserving Approximation of Semiconcave Functions}
\author{Karl Kunisch$^{1,2}$ \\ \texttt{karl.kunisch@uni-graz.at} 
\and Donato Vásquez-Varas$^{2,3,4}$ \\ \texttt{dvasquez@dim.uchile.cl}, \texttt{donato.vasquez-varas@ricam.oeaw.ac.at}} 
\date{
	$^1$Institute of Mathematics and Scientific Computing, University of Graz, Heinrichstraße 36, A-8010 Graz,  Austria\\%
	$^2$Radon Institute for Computational and Applied Mathematics, Austrian Academy of Sciences, Altenbergerstraße 69, A-4040 Linz, Austria
    \\
    $^3$Department of Mathematical Engineering, University of Chile, Beauchef 851, 8370456 Santiago, RM, Chile
    \\ $^4$Center for Mathematical Modelling, Beauchef 851, 8370456 Santiago, RM, Chile
    \\ [2ex]%
	\today
}

\begin{document}
	\maketitle
	
	\begin{abstract}

This article addresses structure-preserving smooth approximation of semiconcave functions. semiconcave functions are of particular interest because they naturally arise in a variety of variational problems, including {optimal feedback control, game theory, and optimal transport}.  
We leverage the fact that any semiconcave function can be represented as the {infimum of a countable family of \(C^2\) functions}. This infimum is expressed in a form that allows {approximation by finitely many functions}, combined with {smoothing operations}, such that each element of the approximating sequence remains semiconcave.  
The {active sets of indices} contributing to the representation of the semiconcave function and its approximations are analyzed in detail. Moreover, we show that the {gradients of the elements in the expansion of the approximating functions form a probability distribution}, a property of particular interest for the {value function in optimal control}.  
Approximation results are established in \(C(\bar \Omega)\) and in \(W^{1,p}(\Omega)\) for \(p \in [1,\infty)\) and \(p = \infty\). Finally, {numerical results} are presented to illustrate the  approach on a test example.


        \bigskip
        
        
	\end{abstract}

	
\section{Introduction}

This work is concerned with the approximation of semiconcave functions. Here we say that for a convex, open, and bounded set $\Omega\subset\R^d$ and $C>0$ a function $v\in C(\overline{\Omega})$ is $C-$semiconcave if the mapping
$$ x\mapsto v(x)-\frac{C}{2}|x|^2$$
is concave. We are particularly  interested in approximations schemes which preserve the semiconcavity structure  and in their asymptotic limits. Our main motivation  for studying these type of functions lies in their connection with control theory. Indeed, under appropriate hypotheses, the value function of an optimal control problem is semiconcave (see \cite{Cannarsa2004} and \cite{Bardi1997}). In the same direction, the semiconcavity is a natural regularity property for certain Hamilton-Jacobi-Bellman equations (see \cite{Kruzkov, Douglis, Flemming}), since it provides uniqueness and it is connected with the developed of maxplus algebra methods (see \cite{Akian,Dower,Gaubert}). Further, Lyapunov semiconcave functions can be used to design discontinuous feedback laws (see \cite{Rifford,Rifford2}) for problems that do not admit smooth ones. In \cite{KuVa2} the existence of a sequence of smooth consistent feedback laws is proved provided that the value functions of the underlying control problem is semiconcave, whereas in \cite{KuVa3} this result is used to prove the convergence of a machine learning scheme for the synthesis of optimal feedback-laws.  Besides optimal control, semiconcave functions also arise in optimal transport \cite{santam}, where they are crucial for the regularity and stability of transport maps, and in the analysis of free boundary value problems, where they help to characterize the structure of the free boundary \cite{Gener}.
Let us also mention that in the literature also the terminology 'weakly-convex' is used,  which would, in analogy to the present setting, be referred to as semiconvex. We mention, for instance, \cite{Goujon} where weakly-convex functions are learned by means of neural networks and used as sparsity promoting regularizers  in image reconstruction problems.

In this paper we develop a method for approximating semiconcave functions with the property that all elements of the approximating family preserve semiconcavity. In order to achieve this, we design a parametrization which corresponds to a regularized version of the minimum of an array of real numbers applied to a family of smooth parametrized functions. Additionally, by means of analyzing the gradients of the regularized  functions, we identify a family of sets in which the convergence of the gradients of the approximation is uniform. Remarkably, this family covers part of the discontinuities of the approximated function. As we will explain, this is particularly relevant for the approximation of solutions to Hamilton Jacobi Bellman equations. 

To the best of our knowledge, semiconcave preserving parametrizations have not been addressed yet. However, there are many contributions to the approximation of convex functions. We point out that many of them use as starting point the fact that convex functions are the maximum of a family of affine functions. In \cite{BoydMagnani}, they are parametrized by the maximum of a finite number of affine functions, which are learned by an alternating optimization method. The convergence of this approach together with an initilization method can be found in \cite{GhoshPananjady} under further assumptions on the initial guess of the parameters. Concerning neural network approaches,  Input Convex Neural Networks (ICNN) are introduced in \cite{AmosXuKolter}. They consist in a fully connected neural network with non-negative weights and increasing convex activation function. In the case of ReLU activation these types of neural networks are universal approximators of globally Lipschitz continuous and convex functions,  see\cite{Hanin}. This approach is extended in \cite{ChenShiZhang}  by Input Convex Recurrent  Neural Network (ICRNN), where a recurrent neural network is considered. The authors used this to estimate the dynamics of a control process in a manner that the resulting approximated control problem is convex. In \cite{Warin} the GroupMax Neural Network architecture is proposed and its universal approximation property with respect to convex functions is proved. Similarly to ICNN, GroupMax Neural Networks utilize  increasing convex  activation functions together with positive weights. The main difference between them is that GroupMax Neural Networks consider group maximization layers.  In the previous cases the approximation is convex and Lipschitz, but it is not smooth. In contrast, in \cite{CalafioreGaubertPossieri} Log-sum-exp neural networks are presented which are based on using the Log-Sum-Exp function to approximate the maximum of a family of affine functions. This architecture is smooth and it enjoys the universal approximation property for convex functions. A closer connection between semiconcavity and neural networks can be found in \cite{DarbonDowerMeng}, where a neural network architecture based on the min-plus algebras is proposed. The approximation is constructed as the minimum of a family of functions which are quadratic in the space variables and also time dependent. This corresponds to a semiconcave approximation of the value function of an optimal control problem.

The organization of the article is as follow. In Section \ref{sec:parametrization} the semiconcavity preserving parametrization is introduced. In Section \ref{sec:ProbInt} a probabilistic interpretation connected with the parametrization is introduced which helps to analyze the behavior of the gradient of the approximation. Section \eqref{sec:approximation} is devoted to investigate the convergence of the approximation. Additionally, in Section \ref{sec:Example} a semi-concave test function is proposed for which
the results of the proposed parametrization are illustrated and compared towards their numerical performance.  
Finally, in Section \ref{sec:conclusion} the main results and conclusions of this work are summarized.

Before continuing, we give some notation. For a vector $x\in\R^d$ we denote its Euclidean norm by $|x|$,  its $p-$norm by $|x|_p=\left(\sum_{i=1}^d|x_i|^p\right)^{\frac{1}{p}}$, for $p\in [1, \infty)$, and the supremum norm by $|x|_\infty=\max_{i\in\{1,\ldots,d\}} |x_i|$. For $A\subset\R^d$ an open set or  the closure of an open set with Lipschitz boundary, and $B\subset  \R^m$, with $X$ a normed vector space equipped with a norm $\norm{\cdot }_{X}$,  we denote the space of continuous functions from $A$ to $B$ by $C(A;B)$ equipped with the norm
$$\norm{v}_{C(A;B)}=\sup_{x\in A}\norm{v(x)}_{X}, $$
where $v\in C(A;B)$. 

For $k\in \N$ we denote the space of functions form $A$ to $B$ which are $k$ times continuously differentiable by $C^k(A;B)$. Additionally, $C_{c}^k(A;B)$ denotes the space of function in $C^k(A;B)$ with compact support in $A$. 

For $A\subset\R^d$ measurable, $B$ a borel measurable subset of a vector normed spaces $X$, and $p\in[1,\infty)$  we denote the space of $p-$integrable functions from $A$ to $B$ by $L^p(A;B)$ and the space of essentially bounded functions from $A$ to $B$ by $L^\infty(A;B)$ equipped with their usual norms.
For $n,m\in\N$ we denote the space of matrices of dimension $n\times m$ by $\R^{n\times m}$ equipped with the operator norm which we denote by $|A|=\sup_{x\in\R^{m}}\frac{|Ax|}{|x|}$ for $A\in\R^{n\times m}$. 

For $A\subset\R^d$ an open set and $B\subset\R^m$ a borel measurable space, we denote by $BV(A;B)$ the space of function $\sigma:U\mapsto B$ of bounded variation equipped with the following norm 
$$ \norm{\sigma}_{BV(U;B)}=\sup_{\begin{array}{c}
\phi\in C_{c}^1(U;B),\\
\norm{\phi}_{C(U;B)}\leq 1
\end{array} }\sum_{i=1}^m\left|\int_{A}\sigma_i \dive(\phi)dx\right|.$$
We also consider the sobolev spaces $W^{1,p}(A)$ for $p\in [1,\infty]$equipped with the following norm:
$$  \norm{v}_{W^{1,p}(A)}=\norm{v}_{L^{p}(A;\R)}+\norm{\nabla v}_{L^p(A;\R^d)}.$$

 We add the sub-index $loc$ to $L^p(A,B)$ to denote the spaces of measurables functions from $A$ to $B$ which are in $L^p(U;B)$ for all $U\subset \R^d$ compactly included in $A$.
We follow the same rule for $BV_{loc}(A;B)$, $W_{loc}^{1,p}(A)$ and $C^k_{loc}(A;B)$.  

In all the spaces defined above, if $B=\R$ or $\R^m=\R$, we omit $B$, for instance we write $C(A)$ instead of $C(A;\R)$.
\section{Semiconcavity preserving parametrization}
\label{sec:parametrization}
In this section we describe the architecture that we propose to approximate semiconcave functions while preserving this semiconcavity property. The approximation is based on the fact that a function is semiconcave if and only if it is the poinwise minimum over a family of $C^2$ functions with uniformly bounded second derivatives. Hence, to approximate  semiconcave functions we consider a smooth approximation of the minimum together with a parameterized family of $C^2(\overline{\Omega})$ functions. To describe the parametrization in more detail we need some concepts and results that we describe in the following. Throughout we consider  $\Omega$ to be an open, bounded, and convex set in $\R^d$.

We start by recalling the definition of a semiconcave function.
\begin{defi}
Let $v:\Omega\mapsto\R$ be a continuous function and $C>0$. We say that $v$ is $C-$semiconcave if $x\in\Omega\mapsto v(x)-C|x|^2$ is concave.
\end{defi}
According to Proposition 1.1.3 in \cite{Cannarsa2004}, a semiconcave function can be characterized as the minimum over a family of $C^2$ function as is stated in the following result. Here we present a slightly modified version, assuming that $v$ is Lipschitz continuous on $\bar \Omega$, rather than on $\Omega$, which is more convenient for our needs.

\begin{theo}
\label{theo:SemiConcaveRepresentation}
 Let  $v\in C(\overline{\Omega})$. Then $v$ is Lipschitz on $\overline{\Omega}$ and  semiconcave with constant $C>0$ if and only if there exists a family of functions $\{\phi_{i}\}_{i\in \mathcal{I}}\subset C^2(\overline{\Omega})$ uniformly bounded in $C^2(\overline{\Omega})$ such that
$$\sup_{i\in \mathcal{I}}\norm{\nabla^2 \phi_i}_{C(\overline{\Omega};\R^{d\times d})}\leq C,$$
  and
\begin{equation}\label{theo:SemiConcaveRepresentation:eq1}
    v(x)=\inf_{i\in \mathcal{I}} \phi _{i}(x) \mbox{ for all }x\in\overline{\Omega},
    \end{equation}
where     $\mathcal{I}$ is an uncountable index set.
\end{theo}

\begin{proof}[Proof of Theorem \ref{theo:SemiConcaveRepresentation}]
Let us first suppose that $v$ satisfies the specified properties  and prove the existence of   a family of functions $\{\phi_{i}\}_{i\in \mathcal{I}}\subset C^2(\overline{\Omega})$ satisfying the conclusions of the theorem. In the following we denote by $L$ the Lipschitz constant of $v$.

Due to the semiconcavity of $v$  and the convexity of $\Omega$, we have by Proposition 3.3.1 in \cite{Cannarsa2004} that for all $x,y\in \Omega$,
\beq p\in D^+v(y) \text{ if and only if } v(x)\leq v(y)+p(x-y)+\frac{C}{2}|x-y|^2, \label{theo:SemiConcaveRepresentation:eq2}\eeq
where $D^+v(y)$ is the upper-differential of $v$ at $y$, which is defined by 
$$D^+v(y):=\left\{p\in\R^d: \ \limsup_{x\to y}\frac{v(x)-v(y)-p\cdot (x-y)}{|y-x|}\leq 0 \right\}.$$ 
according to Definition 3.1.1 in  \cite{Cannarsa2004}. We note that $D^+ v(x)$ is not empty for all $x\in\Omega$ since $v$ is semiconcave (see Theorem 3.3.4 (b) in \cite[Chapter 3]{Cannarsa2004}).

By continuity this inequality also holds for $x\in \bar \Omega$.  We will construct the family $\{\phi_{i}\}_{i\in \mathcal{I}}$ by using \eqref{theo:SemiConcaveRepresentation:eq2}. For this purpose we verify that for every $y\in \Omega$ the norm of the elements in $D^+v(y)$ is bounded by $L$.

Let us consider $y\in\Omega$ and $p\in D^+v(y)$. Without loss of generality we assume $p\neq 0$. Let us set $x=y-\frac{\varepsilon}{|p|}p$. Since $\Omega$ is open, for $\varepsilon$ small enough we have that  $x\in\Omega$. Then for such $\epsilon$ we have by means of \eqref{theo:SemiConcaveRepresentation:eq2}  that
$$ v\left(y-\frac{\varepsilon}{|p|}p\right)\leq v(y)-\varepsilon|p|+\frac{C}{2}\varepsilon^2.$$
Rearranging the terms in the previous inequality and using the Lipschitz continuity of $v$ we obtain
$$ \varepsilon|p| \leq L\varepsilon+\frac{C}{2}\varepsilon^2. $$
Dividing both sides of the inequality by $\varepsilon$ and letting $\varepsilon\to 0^+$ we obtain that $|p|\leq L$, which proves the claim.

For $y\in\partial\Omega$ it is not true that all the elements $D^+v(y)$ are bounded. However, arguing by continuity and using \eqref{theo:SemiConcaveRepresentation:eq2} it is possible to prove that there exists at least one element satisfying \eqref{theo:SemiConcaveRepresentation:eq2} which is bounded by $L$. This observation proves that for every $y\in\overline{\Omega}$ there exists $p_y\in \R^d$ with $|p_y|\leq L$ such that
$$ v(x)\leq v(y)+p_y(x-y)+\frac{C}{2}|y-x|^2.$$
Then, observing that the previous inequality is attained for $x=y$, setting $\mathcal{I}=\overline{\Omega}$, and
$$\phi_y(x)=v(y)+p_y(x-y)+\frac{C}{2}|y-x|^2 \mbox{ for y}\in\overline{\Omega},$$
we obtain \eqref{theo:SemiConcaveRepresentation:eq1}, and the asserted uniform bounds on  $ \{\phi_y\}_{y\in {\bar \Omega}}$.

The reciprocal implication is given by Corollary 2.1.6 in \cite{Cannarsa2004} and the fact that the infimum over a family of Lipschitz functions is Lipschitz as well.

\end{proof}

We next assert that the index set $\mathcal{I}$  in the previous theorem can be replaced by a countable one.

\begin{prop}
\label{prop:discreteRep}
    Let $v\in C(\overline{\Omega})$ be a semiconcave function with constant $C>0$ and Lipschitz continuous in $\overline{\Omega}$ with constant $L>0$. Let  $\{\phi_{i}\}_{i\in\mathcal{I}}$ denote the family  appearing in \Cref{theo:SemiConcaveRepresentation}. Then there exists $\{\tilde{\phi}_i\}_{i=1}^{\infty}\subset \{\phi_{i}\}_{i\in\mathcal{I}}$ such that
     $$ v_{n}(x)=\min_{i=1,\ldots,n}\tilde{\phi}_{i}(x)$$
    satisfies
    \begin{equation}\label{eq:kk3}
    \lim_{n\to\infty}\norm{v_n-v}_{C(\overline{\Omega})}+\norm{\nabla v_n-\nabla v}_{W^{1,p}(\Omega)}=0,
    \end{equation}
    for all $p\in [1,\infty)$, and
     \begin{equation}\label{eq:kk33}
    \lim_{n\to\infty}\nabla v_n(x)= \nabla v(x)\mbox{ a.e. in }\Omega.
    \end{equation}
Further,  for each  compact subset $K$ of $\Omega$, such that $D^+v(x)$ is a singleton for  all $x\in  K$, we have
 \begin{equation}\label{eq:kk333}
 \lim_{n\to\infty}v_n=v\mbox{ in }C^1(K).
  \end{equation}
\end{prop}

\begin{proof}[Proof of Proposition \ref{prop:discreteRep}]
For $n\in \N$ let us consider a family of rectangles $\{\tau_{i,h}\}_{i=1}^{N_{n}}$ such that $diam(\tau_{i,n})\leq \frac{1}{n}$, where $diam(\cdot)$ denotes the diameter of $\tau_{i,n}$, and
$$ \overline{\Omega}\subset \bigcup_{i=1}^{N_n}\tau_{i,n}\mbox{ and }\overline{\Omega}\cap \tau_{i,n}\neq \emptyset \mbox{ for all }i=1,\ldots,N_n.$$
For each $i=1,\ldots,N_n$ we choose $x_{i,n}\in \tau_{i,n}\cap\overline{\Omega}$.  By the definition of infimum,  for each $x_{i,n}$ there exists $j(i,n)$ such that $\phi_{j(i,n)}(x_{i,n})\leq v(x_{i,n})+ \frac{1}{n}$. This combined with the Lipschitz continuity of $v$ on $\overline{\Omega}$  and the uniform boundedness of $\{\phi_i\}_{i\in \mathcal{I}}$  in $C^1(\overline{\Omega})$ we have for each $i= 1,\dots,N_n$
\beq  |\phi_{j(i,n)}(x)-v(x)|\leq \frac{(2 L+1)}{n}, \mbox{ for all }x\in \tau_{i,n}.\label{prop:discreteRep:proof:error1}\eeq
Let us denote by $\tilde{\mathcal{I}}_{n}$ the  finite set composed by the indices $j(i,n)$ defined above. Further set $\mathcal{I}_n=\bigcup_{j\leq n}\tilde{\mathcal{I}}_{j}$ and $\mathcal{I}_{\infty}=\cup_{n\in\N} \mathcal{I}_n$,
and define $v_{n}\in C(\overline{\Omega})$ as follows
$$ v_{n}(x)=\min_{i\in \mathcal{I}_{n}} \phi_i(x).$$
Observe that $v_n$ is semiconcave. Further from \eqref{prop:discreteRep:proof:error1} we deduce that
$$ |v_n(x)-v(x)|\leq \frac{(2 L+1)}{n} \mbox{ for all }x\in \overline{\Omega},$$
from which the uniform convergence of $v_{n}$ to $v$ follows. Almost everywhere convergence of the gradients as claimed in \eqref{eq:kk33} follows from \cite[Theorem 3.3.3]{Cannarsa2004}.

 The convergence of $v_{n}$ in $W^{1,p}(\Omega)$ for all $p\in [1,\infty)$ is a simple consequence of the a.e. convergence of $\nabla v_n$ and the dominated convergence theorem. Consequently \eqref{eq:kk3} is verified.

     We now prove the last part of the result and assume that  $Dv^+(x)$ is a singleton for all $x\in \bar K$. Consequently $v$ is differentiable for all $v\in \bar \Omega$.
      Let us denote by $B \subset \Omega$  the set where the gradients   $v_n$ exists for all $n$, and observe that $|\Omega \setminus B| =0$, where $|\cdot|$ denotes the measure of a set.
     Proceeding by contradiction let us assume that $v_{n}$  does not converge to  $v$ in $W^{1,\infty}(K)$.
       In this case, there exists $\varepsilon>0$, a sequence $x_k\in B $, and a sub-sequence of $v_{n}$ denoted by $v_{n(k)}$ such that
     \beq  |\nabla v_{n(k)}(x_k)-\nabla v(x_k)|\geq \varepsilon.\label{prop:gradConv:proof:eq1}\eeq
     Since $K$ is compact and $\{v_{n}\}_n$ is uniformly bounded in $W^{1,\infty}(\Omega)$, there exists $\bar{x}$ in $\bar K$ and $p\in\R^d$ such that $x_k$ converges to $\bar{x}$ and $\nabla v_{n(k)}$ converges to $p$. By the semiconcavity of $v_n$ is easy to verify that
     $$ v(y)-v(\bar{x})-p\cdot(y-\bar{x})\leq C|\bar x-y|^2, \text { for all } y \in \Omega.$$
     This implies that $p\in D^+ v(\bar{x})$ and by assumption $p= \nabla v(\bar{x})$.
      However, by \eqref{prop:gradConv:proof:eq1} we  have
       $| p-\nabla v(\bar{x})|\ge \epsilon.$
     This  gives the desired  contradiction, and  $\lim_{n\to\infty}v_n=v\mbox{ in }C^1(K)$ follows.
     \end{proof}

Our parametrization will  build on this result. It  consists in replacing the infinite family $\{\tilde \phi_{i}\}_{i=1}^\infty$ by a family of $n$ parameterized functions and  it utilizes a smooth approximation of the minimum operation $\psi_n$ of $n$ real numbers. We first describe the construction of the smooth approximation $\psi_{n,\varepsilon}$, $\varepsilon>0$,  and its properties. Then, relying on these properties, we introduce the family of parameterized functions  together with the usage of $\psi_{n,\varepsilon}$, see \eqref{def:parametrization} below.

For $n\le m\in\N\setminus\{0\}$ and $a\in \R^m$ let us denote by $\psi_{n}(a)$ the minimum over $a$, that is,
\beq
\psi_{n}(a)=\min_{i\in\{1,\ldots,n\}}a_i.
 \eeq
 Thus, if $n < m$ then $\psi_{n}(a)$ only considers the first $n$ elements of $a$.  We also observe that $\psi_{n}$ is a $1-$Lipschitz continuous function on $\R^m$ endowed with the maximum norm, but is not $C^1$.

For obtaining a $C^1$ approximation of $\psi_n$, we note that it can be written in a recursive manner, namely, for $i\in \{1,\ldots,n-1\}$ we have the following relation
       \begin{equation}\label{eq:kk2}
        \psi_{i+1}(a)=\min\{a_{i+1},\psi_{i}(a)\}, \mbox{ with }\psi_{1}=a_1.
        \end{equation}
Further, for the case of the minimum of only two elements $x,y\in\R$ we have
\beq \min(x,y)= x-(x-y)_+.\label{eq:min2}\eeq
where $(\cdot)_+$ stands for the positive part function. Combining \eqref{eq:kk2} and \eqref{eq:min2}, we can evaluate $\psi_{n}$ by the following recursive formula
\beq \psi_{i+1}(a)=a_{i+1}-(a_{i+1}-\psi_{i}(a))_+, \ \psi_{1}=a_1, \ i\in\{1,\ldots,n-1\}.\label{def:Psi}\eeq
It is noteworthy that an equivalent formula holds for the case of the maximum of a vector and this allowed the authors of \cite{Hanin} and \cite{ChenShiZhang} to prove the universal approximation property of ICNN and ICRNN. In our case we will use this to provide a smooth approximation of the minimum.

We shall employ  a smooth regularization $g_{\varepsilon}\in C^{1,1}(\R),$  $\varepsilon>0$,  of the positive part function $(\cdot)_+$, which satisfies the following properties

\begin{subequations}
\label{hypo:positivepart}
\begin{align}
g_{\varepsilon}(x)\geq 0 \text{ for all }x\in\R,
    \label{hypo:positivepart:1}\\
     g_{\varepsilon}'(x)\in [0,1]\text{ for all }x\in \R,\label{hypo:positivepart:2} \\
     g_{\varepsilon}''(x)\geq 0\text{ for all {{ almost all }} }x\in\R,\label{hypo:positivepart:3} \\
     \norm{(\cdot)_{+}-g_{\varepsilon}}_{C(\R)}\leq \varepsilon,\label{hypo:positivepart:4} \\
     \lim_{\varepsilon\to 0^+}g_{\varepsilon}'=\chi_{[0,\infty)}\mbox{ in }  C_{loc}(\R\setminus\{0\}), \label{hypo:positivepart:5}\\
     g_{\varepsilon}''\in L^\infty(\mathbb{R}) \text{ for each } \varepsilon >0. \label{hypo:positivepart:6}
\end{align}
\end{subequations}

\begin{rem}\label{rem:kk1}
We provide examples of  functions $g_{\varepsilon}$ which \eqref{hypo:positivepart}. We shall return to them in the following section.  First, we consider the Moreau envelope of the positive part function, namely,
$$ g_{\varepsilon,M}(s)=\min_{t\in \R}(t)_{+}+\frac{1}{2\varepsilon}|t-s|^2=\left\{
\begin{array}{ll}
0 & \mbox{ if }s<0\\
\frac{1}{2\varepsilon}s^2 & \mbox{ if } s\in [0,\varepsilon) \\
s-\frac{\varepsilon}{2} & \mbox{ if }s\geq\varepsilon
\end{array}
\right. $$
It is not hard to see that $g_{\varepsilon,M}$ satisfies \eqref{hypo:positivepart}. Further, it satisfies that
$ g_{\varepsilon,M}'(0)=0 $ for all $\varepsilon>0$ and $g_{\varepsilon,M} \in C^{1,1}(\R)$, but it is not in $C^{2}(\R)$.

Another example is $g_{\varepsilon,A}$ defined as:
$$g_{\varepsilon,A}(s)=\frac{1}{2}\left(s+\sqrt{s^2+\varepsilon^2}-\varepsilon\right).$$
To see that this function satisfies \eqref{hypo:positivepart} we note that $(s)_+=\frac{1}{2}(|s|+s)$ and that $\sqrt{s^2+\varepsilon^2}-\varepsilon$ is a smooth approximation of the absolute value. Using these properties it is easy to see that $g_{\varepsilon,A}$ satisfies \eqref{hypo:positivepart}.
\end{rem}

 By replacing the positive part in \eqref{def:Psi} by $g_{\varepsilon}$ we obtain a smooth approximation of $\psi_{n}$ given by the following  process:
\beq  \psi_{i+1,\varepsilon}(a)=a_{i+1}-g_{\varepsilon}(a_{i+1}-\psi_{i,\varepsilon}(a)),\ i\in\{1,\ldots,n-1\},\ \psi_{1,\varepsilon}=a_1.
\label{defi:PsiAp}
\eeq

In Proposition \ref{prop:propertiesApprox} below, some useful properties of $\psi_{n,\varepsilon}$ are shown. These properties will permit us to construct a semiconcavity preserving parametrization.

\begin{prop}
\label{prop:propertiesApprox}
Let $n\in\N\cap[2,\infty)$ and $\varepsilon>0$ be arbitrarily  fixed, and assume that $g_{\varepsilon}\in C^{1,1}({\R})$ satisfies \eqref{hypo:positivepart}. Then the function $\psi_{n,\varepsilon}$ defined in \eqref{defi:PsiAp} is of class $C_{loc}^{1,1}(\R^n)$ and satisfies for all $i,j\in\{1,\ldots,n\}$ that
\beq
\frac{\partial\psi_{n,\varepsilon}}{\partial a_i}(a)=\delta_{i,n}-g'_{\varepsilon}(a_n-\psi_{n-1,\varepsilon}(a))(\delta_{i,n}-\frac{\partial\psi_{n-1,\varepsilon}}{\partial a_i}(a)), \mbox{ for all }a\in\R^n,
\label{prop:propertiesApprox:cons1}
\eeq
\beq
\begin{array}{l}
\displaystyle\frac{\partial^2 \psi_{n,\varepsilon}}{\partial a_i,\partial a_j}(a)=-g''_{\varepsilon}(a_n-\psi_{n-1,\varepsilon}(a))\left(\delta_{i,n}-\frac{\partial\psi_{n-1,\varepsilon}}{\partial a_i}(a)\right)\left(\delta_{j,n}-\frac{\partial\psi_{n-1,\varepsilon}}{\partial a_j}(a)\right)\\
\ecart\displaystyle
+g'_{\varepsilon}(a_n-\psi_{n-1,\varepsilon}(a))\frac{\partial^2 \psi_{n-1,\varepsilon}}{\partial a_i,\partial a_j}(a), \mbox{ for almost all }a\in\R^n,
\end{array}
\label{prop:propertiesApprox:cons2}
\eeq
where $\delta_{i,j}$ stand for the Kronecker delta. Further, we have that
\beq
\norm{\psi_{n,\varepsilon}-\psi_{n}}_{L^{\infty}(\R^n)}\leq (n-1)\varepsilon,
\label{prop:propertiesApprox:cons6}
\eeq

\beq \frac{\partial\psi_{n,\varepsilon}}{\partial a_i}(a)\geq 0,\quad \sum_{j=1}^{n}\frac{\partial\psi_{n,\varepsilon}}{\partial a_j}(a)=1,\mbox{ for all }a\in\R^n
\label{prop:propertiesApprox:cons3}
\eeq

\beq
-2(n-1)\norm{g''_\varepsilon}_{L^\infty(\R)}\leq b^{\top}\nabla^2\psi_{n,\varepsilon}(a)b \leq 0,\mbox{ for almost all }a\in\R^n,
\label{prop:propertiesApprox:cons5}
\eeq
for all $b\in\R^n$ with $|b|=1$, and
\beq
\left|\nabla^2 \psi_{n,\varepsilon}(a)\right|\leq 2(n-1)\norm{g''_\varepsilon}_{L^\infty(\R)}
\mbox{ for almost all }a\in\R^n.
\label{prop:propertiesApprox:cons4}
\eeq

\end{prop}
\begin{proof}
By the chain rule it is immediate to obtain \eqref{prop:propertiesApprox:cons1} and \eqref{prop:propertiesApprox:cons2}. In \eqref{prop:propertiesApprox:cons3}, the positivity of the partial derivatives is a direct consequence of \eqref{hypo:positivepart:2}. For the second part of \eqref{prop:propertiesApprox:cons3} we proceed by induction on $n$. The base case is $\psi_{1,\epsilon}(a)=a_1$, which clearly holds true. Let us assume as inductive hypothesis that the claim holds for $n-1$ with $n\geq 2$. By \eqref{prop:propertiesApprox:cons1} and again \eqref{hypo:positivepart:2},  we see that $\nabla \psi_{n.\varepsilon}(a)$ is a convex combination between $e_{n}$ (the $n-th$ vector of the canonical basis of $\R^n$) and $\nabla \psi_{n-1}(a)$, where the gradient is with respect to $a$. Since the sum of the $e_{n}$ is equal to $1$ and by the inductive hypothesis the same holds true for $\nabla \psi_{n-1}(a)$, we get that the sum of the elements of $\nabla \psi_{n}(a)$ is $1$ as well, which proves the second part of \eqref{prop:propertiesApprox:cons3}.

Let $b\in\R^n$ be such that $|b|=1$. By \eqref{prop:propertiesApprox:cons2} we have that

\begin{equation}
\begin{array}{l}
\displaystyle
\sum_{i,j=1}^n \frac{\partial^2\psi_{n,\varepsilon}(a)}{\partial a_{i}\partial a_{j}}b_ib_j =-g''_{\varepsilon}(a_n-\psi_{n-1,\varepsilon}(a))((e_n-\nabla \psi_{n-1,\varepsilon}(a))\cdot b)^2\\
\ecart \displaystyle
+g'_{\varepsilon}(a_n-\psi_{n-1,\varepsilon}(a)) \sum_{i,j=1}^n \frac{\partial^2\psi_{n-1,\varepsilon}(a)}{\partial a_{i}\partial a_{j}}b_ib_j.\label{prop:propertiesApprox:eq1}
\end{array}
\end{equation}

We note that the first term in the right hand-side of the above expression is bounded from below. To see this, we first point out that
\[((e_n-\nabla\psi_{n-1,\varepsilon}(a))\cdot b)^2=\left(b_n-\sum_{i=1}^{n-1}\frac{\partial \psi_{n-1,\varepsilon}}{\partial a_i}(a)b_i\right)^2\leq 2 \left(b_{n}^2+\sum_{i=1}^{n-1}\frac{\partial \psi_{n-1,\varepsilon}}{\partial a_i}(a)b_i^2\right)\le 2,\]
where we have used \eqref{prop:propertiesApprox:cons3}, Jensen inequality, $|b|=1$, and the fact that $(x-y)^2\leq 2(x^2+y^2)$ for all $x,y\in\R$. Using this and \eqref{hypo:positivepart:3} in \eqref{prop:propertiesApprox:eq1} we obtain
\begin{equation*}
\begin{array}{l}    
\displaystyle
-2g''_{\varepsilon}(a_n-\psi_{n-1,\varepsilon}(a))  +  g'_{\varepsilon}(a_n-\psi_{n-1,\varepsilon}(a))\sum_{i,j=1}^n \frac{\partial^2\psi_{n-1,\varepsilon}(a)}{\partial a_{i}\partial a_{j}}b_ib_j\\
\ecart \displaystyle
 \leq \sum_{i,j=1}^n \frac{\partial^2\psi_{n,\varepsilon}(a)}{\partial a_{i}\partial a_{j}}b_ib_j\leq \sum_{i,j=1}^n \frac{\partial^2\psi_{n-1,\varepsilon}(a)}{\partial a_{i}\partial a_{j}}b_ib_j.
\end{array}
\end{equation*}
Using induction in both of the above inequalities we obtain \eqref{prop:propertiesApprox:cons5}.

This implies that for almost all $a\in\R^n$, the eigenvalues of $\nabla^2 \psi_{n,\varepsilon}(a)$ are contained 
in $[-2(n-1)\norm{g''_{\varepsilon}}_{L^\infty(\R)},0]$, and hence \eqref{prop:propertiesApprox:cons4} holds.

We turn now our attention to the proof of \eqref{prop:propertiesApprox:cons6}. By subtracting \eqref{def:Psi} from \eqref{defi:PsiAp} we have that for all $a\in\R^n$ and $i=1,\ldots,n-1$
    $$ \psi_{i+1,\varepsilon}(a)-\psi_{i+1}(a)=
    \left(a_{i+1}-\psi_{i}(a)\right)_+-g_{\varepsilon}
    \left(a_{i+1}-\psi_{i,\varepsilon}(a)\right). $$
    Subtracting and adding  the term $\left(a_{i+1}-\psi_{i,\varepsilon}(a)\right)_+$ on the right hand-side of the previous inequality we obtain
    $$ \begin{array}{l}
         \displaystyle \psi_{i+1,\varepsilon}(a)-\psi_{i+1}(a)=  \left(a_{i+1}-\psi_{i}(a)\right)_+-\left(a_{i+1}-\psi_{i,\varepsilon}(a)\right)_++\\
         \ecart\displaystyle \left(a_{i+1}-\psi_{i,\varepsilon}(a)\right)_+-g_{\varepsilon}\left(a_{i+1}-\psi_{i,\varepsilon}(a)\right).
    \end{array} $$
    Using the 1-Lipschitz continuity of the positive part and  \eqref{hypo:positivepart:4} in the above expression we obtain that
    $$ |\psi_{i+1,\varepsilon}(a)-\psi_{i+1}(a)|\leq \varepsilon +|\psi_{i,\varepsilon}(a)-\psi_{i}(a)|.$$
    Summing from $i=1$ to $i=n-1$ we obtain the desired estimate for all $a \in \R^n$
    $$ |\psi_{n,\varepsilon}(a)-\psi_{n}(a)|\leq  (n-1)\varepsilon.$$
\end{proof}
\begin{rem}
It is interesting to observe that \eqref{prop:propertiesApprox:cons3} implies that the partial derivative of $\psi_{n,\varepsilon}$ at a point in $\R^n$ form a probability distribution. We will delve into this in Section \ref{sec:ProbInt}, where via \eqref{prop:propertiesApprox:cons3}, we give a  probabilistic interpretation of the partial derivatives of $\psi_{n,\varepsilon}$ and connect it with the discontinuities of the gradient of a semiconcave function. 
\end{rem}

To construct an approximation of the type \eqref{theo:SemiConcaveRepresentation:eq1} we choose a finite family of functions   $\{\phi_{i}\}_{i=1}^n\subset C^2(\overline{\Omega})$ and set
\begin{equation}\label{eq:kk1}
\Phi_n=(\phi_1,\ldots,\phi_n), \quad v_{n}=\psi_{n}\circ\Phi_n, \quad \text{ and } v_{n,\varepsilon}=\psi_{n,\varepsilon}\circ\Phi_{n}.
\end{equation}
In Proposition \ref{lemma:pospartapprox} below, we prove that $v_{n,\varepsilon}$ is a $C^{1,1}(\overline{\Omega})$ approximation of $v_n$ and that it is semiconcave.

\begin{prop}
\label{lemma:pospartapprox}
Let $\{\phi_{i}\}_{i=1}^n\subset C^2(\overline{\Omega})$.
Then $v_{n,\varepsilon}$ is $C-$semiconcave and $L-$Lipschitz continuous, with $$
C=\max_{i\in\{1,\ldots,n\}}\norm{\nabla^2 \phi_i}_{C(\overline{\Omega};\R^{d\times d})}\mbox{ and }L=\max_{i\in\{1,\ldots,n\}}\norm{\nabla \phi_i}_{C(\overline{\Omega};\R^{d})}.$$
Further,
\beq
\norm{v_{n,\varepsilon}-v_n}_{C(\overline{\Omega})}\leq (n-1)\varepsilon.
\label{lemma:pospartapprox:eq1}\eeq
and $v_{n,\varepsilon}$ is of class $C^{1,1}$ with
\beq
\norm{\nabla^2 v_{n,\varepsilon}}_{L^\infty(\Omega;\R^{d\times d})}\leq C+2n(n-1)L^2\norm{g''_{\varepsilon}}_{L^\infty(\R)}.
\label{lemma:pospartapprox:eq3}
\eeq
\end{prop}
\begin{proof}
We note that $v_{n,\varepsilon}$ is in $C^{1,1}(\overline{\Omega})$ since it is the composition of $\psi_{n,\varepsilon}\in C^{1,1}(\R^n)$ and $\Phi_n\in C^2(\overline{\Omega};\R^n)$. For $x\in\overline{\Omega}$, by the chain rule we have that
$$
 \nabla v_{n,\varepsilon}=\sum_{i=1}^n\frac{\partial\psi_{n,\varepsilon}}{\partial a_i}(\Phi_n(x))\nabla \phi_i(x),$$ and by \eqref{prop:propertiesApprox:cons3} with $a= \Phi_{n,\epsilon}(x)$ we find that
\begin{align}    
 \left| \nabla v_{n,\varepsilon}(x)\right|&=\left|\sum_{i=1}^n\frac{\partial\psi_{n,\varepsilon}}{\partial a_i}(\Phi_n(x))\nabla \phi_i(x)\right|\\[1.3ex] 
 &\leq \sum_{i=1}^n\frac{\partial\psi_{n,\varepsilon}}{\partial a_i}(\Phi_n(x))\norm{\nabla \phi_{i}}_{C(\overline{\Omega};\R^d)}\leq \max_{i\in\{1,\ldots,n\}}\norm{\nabla \phi_{i}}_{C(\overline{\Omega};\R^d)}.
 \end{align}
The asserted Lipschitz continuity of  $v_{n,\varepsilon}$ follows from this estimate.
We next argue the semiconcavity of $v_{n,\varepsilon}$ and the bound \eqref{lemma:pospartapprox:eq3}.   For this purpose, let us consider $y\in\R^d$. By the chain rule we have for almost all $x\in\Omega$ and all $i,j\in\{1,\ldots,d\}$ that
\beq
\frac{\partial v_{n,\varepsilon}}{\partial x_i\partial x_j}(x)=\sum_{k,r=1}^n\frac{\partial^2\psi_{n,\varepsilon}}{\partial a_r\partial a_k}(\Phi_n(x))\frac{\partial\phi_r}{\partial x_i}(x)\frac{\partial\phi_k}{\partial x_j}(x)+\sum_{k=1}^n\frac{\partial \psi_{n,\varepsilon}}{\partial a_k}(\Phi_n(x))\frac{\partial^2 \phi_{k}}{\partial x_i\partial x_j}(x).
\label{lemma:pospartapprox:eq2}
\eeq
We note that by \eqref{prop:propertiesApprox:cons5}, we have for almost all $x\in\Omega$ and all $y\in\R^d$, with $|y|=1$, that
\begin{equation*}
\begin{array}{l}
\displaystyle
-2n(n-1)\norm{g''}_{L^\infty(\R)}\max_{i\in\{1,\ldots,n\}}\norm{\nabla \phi_i}_{C(\overline{\Omega};\R^d)}^2\\
\ecart\displaystyle
\leq \sum_{i,j=1}^d\sum_{k,r=1}^n\frac{\partial^2\psi_{n,\varepsilon}}{\partial a_r\partial a_k}(\Phi_n(x))\frac{\partial\phi_r}{\partial x_i}(x)\frac{\partial\phi_k}{\partial x_j}(x)y_iy_j
=\sum_{k,r=1}^n z_r(x) \frac{\partial^2\psi_{n,\varepsilon}}{\partial a_r\partial a_k}(\Phi_n(x)) z_k(x) \leq 0,
\end{array}
\end{equation*}
where $z_r(x)= \sum_{i=1}^d \frac{\partial\phi_r}{\partial x_i}(x) y_i$,  $z_k(x)= \sum_{j=1}^d \frac{\partial\phi_k}{\partial x_j}(x) y_j$.
Combining this with \eqref{lemma:pospartapprox:eq2}, \eqref{prop:propertiesApprox:cons3}, and  \eqref{prop:propertiesApprox:cons4} , we obtain
$$ -C-2n(n-1)\norm{g''}_{L^\infty(\R)}L^2\leq y^{\top}\cdot\nabla^2v_{n,\varepsilon}(x)\cdot y \leq C, \text{ for all } y\in \R^d, \text{ with } |y|=1.$$
This implies that the eigenvalues of $\nabla^2v_{n,\varepsilon}(x)$ are in $[-C-2n(n-1)\norm{g''}_{L^\infty(\R)}L^2,C]$ from which we deduce \eqref{lemma:pospartapprox:eq3},  and that the semiconcavity constant of $v_{n,\varepsilon}$ is bounded by $C$, see Proposition 1.1.3 in \cite{Cannarsa2004}.

To conclude the proof, we note that \eqref{lemma:pospartapprox:eq1} is a direct consequence of \eqref{prop:propertiesApprox:cons6}.
\end{proof}

We are now in position to introduce our semiconcavity preserving parametrization. With Proposition \ref{prop:discreteRep} and \eqref{eq:kk1} in mind, for the purpose of numerical realization, it remains to choose an approximation of the family of functions $\phi_i$ in order achieve a finite dimensional semiconcave parametrization for functions $v$. For this purpose, let us consider a finite dimensional Banach space $\Theta$ equipped with a norm $\norm{\cdot}_{\Theta}$ and a continuous function $\xi:\Theta\mapsto C^2(\overline{\Omega})$. For $n\in\N$ and $\theta=(\theta^1,\ldots,\theta^n)\in \Theta^n$, we set $\Xi_{n}(\theta)=(\xi(\theta^1),\ldots,\xi(\theta^n))$. Thus, for $\theta\in\Theta^n$ we have that $\Xi_n(\theta)$ acts as a family of smooth functions which  parameterizes $\Phi$.  We equip $\Theta^n$ with the supremum norm, namely, for $\theta\in\Theta^n$ we use  $\norm{\theta}_{\Theta^n}=\sup_{i=1,\ldots,n}\norm{\theta_i}_{\Theta}$. We call the tuple  $(\Theta,\xi)$ a setting.

We propose the following semiconcavity preserving parametrization
\beq  \varv_{n,\varepsilon}(\theta)=\psi_{n,\varepsilon}\circ \Xi_n(\theta),\label{def:parametrization}\eeq
with $\varepsilon\geq 0.$ If $\varepsilon=0$ we drop the $\varepsilon$ from the sub-index and write
$$\varv_{n}(\theta)=\psi_{n}\circ \Xi_n(\theta).$$
By \Cref{lemma:pospartapprox} it is clear that for $\theta\in\Theta^n$ fixed the  semiconcavity constant of $\varv_{n,\varepsilon}(\theta)$ will be bounded by
\beq
    \max_{i\in\{1,\ldots,n\}}\norm{\nabla ^2\xi(\theta^i)}_{C(\overline{\Omega};\R^{d\times d})}.
\label{SemiconcavityConsApprox}
\eeq
Therefore, controlling this quantity is crucial for preserving  semiconcavity.

In section \ref{sec:approximation} the convergence properties of $\varv_{n,\varepsilon}$ are studied with respect to a sequence of settings $\{(\Theta^m,\xi^m)\}_{m\in\N}$ as $\varepsilon\to 0^+$ and $n\to\infty$.

\section{Probabilistic interpretation and active sets}

\label{sec:ProbInt}
In this section we give a probabilistic interpretation of the parametrization introduced in \Cref{sec:parametrization}. To this end, we consider a family of functions $\{\phi_i\}_{i=1}^n\subset C^2(\overline{\Omega})$ and the functions $v_n$ and $v_{n,\varepsilon}$ as in \eqref{eq:kk1}. The interpretation consists in setting for each $a\in\R^n$ a probability $\{p_{n,i,\varepsilon}(a)\}_{i=1}^{n}$ over the indices $\{1,\ldots,n\}$ by using the partial derivatives of $\psi_n$. We will see that in the limit as $\varepsilon\to 0^+$, the probability $\{p_{n,i,\varepsilon}(a)\}_{i=1}^{n}$ has support on the active set of indices $\{i\in\N:\ 1\leq i \leq n,\ \psi_n(a)=a_i  \}$ and we will provide hypotheses under which an explicit formula for the limit probabilities exists. This will help to characterize the \emph{active sets} for $v_{n,\varepsilon}$ and $v_n$ which are the sets where $\phi_i$ is equal to $v_n$,  and where $\nabla \phi_i$ contributes the most to the gradient of $v_n$.

The relevance of this probabilitic interpretation and the active sets described above lies in the fact that the discontinuities of $\nabla v_n$ occur in the boundary of the active sets, thus understanding the behavior of the probabilities and the actives sets is important regarding the characterization of the discontinuities of the gradient of $v_n$.  Further, a correct representation of the active sets by the approximation $v_{n,\varepsilon}$ is crucial for capturing the correct behavior of $v_{n}$ and its gradient around the discontinuities. 

The behavior of the gradients  is particularly importance if
$v_n$ represents the value function of an optimal control problem. In this case, the optimal control in feedback form can be expressed by means of the gradient of the value function. We explain this in more details in  Remark \ref{rem:prob:HJB}.

We commence by observing  that
\beq p_{n,j,\varepsilon}(a)=\frac{\partial \psi_{n,\varepsilon}}{\partial a_i}(a) \text{ for  } a \in \R^n \text{ and }j\in\{1,\ldots,n\}.
\label{def:prob}
\eeq
By\eqref{prop:propertiesApprox:cons3} the family $\{ p_{n,j,\varepsilon}(a) \}_{j=1}^n$ can be interpreted as a probability distribution over the indices. Furthermore, by the chain rule we have for $x\in\Omega$ that
\begin{equation}\label{eqkk4}
 \nabla v_{n,\varepsilon}(x)=\sum_{i=1}^n \frac{\partial\psi_{n,\varepsilon}}{\partial a_i}(\Phi_n(x))\nabla \phi_i(x)=\sum_{i=1}^n p_{n,i,\varepsilon}(\Phi_n(x))\nabla \phi_i(x),
 \end{equation}
from where we see that $p_{n,i,\varepsilon}(\Phi_n(x))$ measure the contribution of $\nabla \phi_i(x)$ in the summation and also that $\nabla v_{n,\varepsilon}(x)$ correspond to the average of $\{\nabla \phi_i(x)\}_{i=1}^{n}$ weighted by $p_{n,i,\varepsilon}(\Phi_n(x))$.

It is noteworthy that by \eqref{defi:PsiAp} the following recursive relations hold
\beq
 p_{m,j,\varepsilon}(a)=\delta_{m,j}-g_{\varepsilon}'\left(a_m-\psi_{m-1,\varepsilon}(a)\right)\left(\delta_{m,j}
 -p_{m-1,j,\varepsilon}(a)\right)\text{ for }j\in\{1,\ldots,n\}\text{ and }m\in \{2,\ldots,n\}
\eeq
and 
\beq 
 p_{1,j,\varepsilon}(a)=\delta_{1,j} \text{ for }j\in\{1,\ldots,n\},
\eeq
for $a\in\R^n$, where it is understood that $p_{i,j,\varepsilon}=0$ for $i\in\{1,\ldots,n-1\}$ and $j\in\{i+1,\ldots,n\}$.

 We can iterate these relations to obtain
\beq 
p_{m,j,\varepsilon}(a)=\left(\prod_{i=j+1}^{m}g'_{\varepsilon}\left(a_{i}-\psi_{i-1,\varepsilon}(a)\right)\right)\left(1-g'_{\varepsilon}(a_j-\psi_{j-1,\varepsilon}(a))\right)
\label{eq:recursive:1}
\eeq 
for all $m\in\{2,\ldots,n\}$ and $j\in \{1,\ldots,m-1\}$, and 
\beq 
p_{j,j,\varepsilon}(a)=\left(1-g'_{\varepsilon}(a_j-\psi_{j-1,\varepsilon}(a))\right).
\label{eq:recursive:2}
\eeq 
Next we address the limiting behavior of the probability distributions $p_{n,j,\varepsilon}(a)$  as $\varepsilon$ tends to $0$. For this will make use of the following assuption on the family of functions $\{g_{\varepsilon}\}_{\varepsilon>0}$ which approximates the positive part:
\beq g_{\varepsilon}(r)\leq (r)_+\mbox{ for all }r\in\R, \label{lemma:prob:hyp0}\eeq
and 
\beq\text{ there exits }s_0\text{ satisfying }\lim_{\varepsilon\to 0^+}\sup_{s\in [-s_0,0]}|g'_{\varepsilon}(s)|=0.\label{lemma:prob:hyp1}\eeq

For this we will need the following technical lemma:
\begin{lemma}
\label{lemma:upperboundPsi}
Let $\{g_{\varepsilon}\}_{\varepsilon>0}$ be a family of $C^{1,1}(\R)$ functions satisfying \eqref{hypo:positivepart} and \eqref{lemma:prob:hyp0}. Then for all $m\in\{1,\ldots,n\}$ we have
\beq
\psi_{m}(a)\leq\psi_{m,\varepsilon}(a) \mbox{ for }a\in\R^n\mbox{ and all }\varepsilon\in(0,\infty).
\eeq
\end{lemma}
\begin{proof}
We proceed by induction. Let us consider $a\in\R^n$ fixed. The base case $m=1$ is trivial since $\psi_{1,\varepsilon}(a)=a_1=\psi_1(a)$. Let us assume that the claim holds for $m\in\{2,\ldots,n-1\}$, that is, $\psi_{m}(a)\leq \psi_{m,\varepsilon}(a)$ . By \eqref{lemma:prob:hyp0} we have that
$$ \psi_{m+1,\varepsilon}(a)\geq a_{m+1}-(a_{m+1}-\psi_{m,\varepsilon}(a))_+.$$
By the monotonicity of the positive part we have that $(a_{m+1}-\psi_{m,\varepsilon})_+\leq (a_{m+1}-\psi_{m})_+$, due to the fact that $\psi_{m}(a)\leq \psi_{m,\varepsilon}(a)$. From this we deduce that $\psi_{m+1,\varepsilon}(a)\geq \psi_{m+1}(a)$, which proves the result.

\end{proof}

In order to simplify the notation in what follows, for $a\in\R^n$, we define the set of active indices  $I_{n}(a)\subset\{1,\ldots,n\}$ as the argmin over the family $\{{\phi_i(a)}\}_{i=1}^n$, and $\hat{i}_{n}(a)\in \{1,\ldots,n\}$ as the largest active index, that is, 
\beq 
I_{n}(a)=\{i\in\{1,\ldots,n\}: \psi_{n}(a)=a_i\},\text{ and }\hat{i}_{n}(a)=\max I_n(a).
\eeq

\begin{prop} \label{lemma:prob}
Let $g_{\varepsilon}\in C^{1,1}(\R)$ with $\varepsilon >0$   be a family of functions satisfying \eqref{hypo:positivepart}. Then for $a\in \R^n$ and all $j\notin I_{n}(a)$ we have
\beq
\lim_{\varepsilon\to 0^+}p_{n,j,\varepsilon}(a)=0.
\label{lemma:prob:cons1}
\eeq
If $I_{n}(a)$ is a singleton, i.e. $I_{n}(a)=\{\hat{i}_{n}(a)\}$, then
\beq \lim_{\varepsilon\to 0^+}p_{n,\hat{i}_{n}(a),\varepsilon}(a)=1.
 \label{lemma:prob:cons3}
\eeq
Further, if $g_{\varepsilon}$ satisfies \eqref{lemma:prob:hyp0} and \eqref{lemma:prob:hyp1}, then for all $j\in\{1,\ldots,n\}$ we have
\beq  \lim_{\varepsilon\to 0^+}p_{n,j,\varepsilon}(a)=\left\{
\begin{array}{ll}
1 & \mbox{if }j=\hat{i}_{n}(a)
 \\
0 & \mbox{otherwise}.
\end{array}
\right. \label{lemma:prob:cons2}\eeq
\end{prop}

\begin{proof}
Let us start by proving \eqref{lemma:prob:cons1} and choose some $j \in  \{1,\ldots,n\}\setminus I_{n}(x)$. We consider separately the cases $a_j\leq \psi_{j-1}(a)$ and $a_j> \psi_{j-1}(a)$. In the first case, since $j\notin I_n(a)$, there exist $i\in\{j+1,\ldots,n\}$ such that $a_i<\psi_{i-1}(a)$, otherwise we have that $a_j=\psi_{n}(a)$ which is a contradiction. Then by \eqref{prop:propertiesApprox:cons6} and \eqref{hypo:positivepart:5}, we have 
$$\lim_{\varepsilon\to 0^+}g'_{\varepsilon}(a_i-\psi_{i-1,\varepsilon}(a))=0, $$
which by means of \eqref{eq:recursive:1} implies that 
$\lim_{\varepsilon\to 0^+}p_{n,j,\varepsilon}(a)=0.$
On the other hand, if   $a_j> \psi_{j-1}(a)$, there exists $\delta>0$ such that \beq a_j-\psi_{j-1}(a)>\delta.
\label{lemma:prob:eq10}
\eeq By \eqref{prop:propertiesApprox:cons6} we have that for all $\varepsilon\in \left(0,\frac{\delta}{2(j-2)}\right)$ the estimate
\beq |\psi_{j-1}(a)-\psi_{j-1,\varepsilon}(a)|\leq \frac{\delta}{2}.
\label{lemma:prob:eq11}
\eeq
Combining \eqref{lemma:prob:eq10} and \eqref{lemma:prob:eq11},
 we arrive at 
\beq \frac{\delta}{2}\leq a_j-\psi_{j-1,\varepsilon}(a)<a_j-\psi_{j-1}(a)+\frac{\delta}{2},\text{ for all }\varepsilon\in \left(0,\frac{\delta}{2(j-2)}\right).
\label{lemma:prob:eq12}
\eeq
This implies that $\{a_j-\psi_{j-1,\varepsilon}(a)\}_{\varepsilon\in J}$, with $J=\left(0,\frac{\delta}{2(j-2)}\right)$ lies in a compact subset of $\R\setminus\{0\}$. Using this, \eqref{prop:propertiesApprox:cons6}, and \eqref{hypo:positivepart:5}, we get that 
$$\lim_{\varepsilon\to 0^+}g'_{\varepsilon}(a_j-\psi_{j-1,\varepsilon}(a))=\chi_{[0,\infty)}(a_j-\psi_{j-1}(a))=1. $$
Combining this with  \eqref{eq:recursive:1}, we again obtain \eqref{lemma:prob:cons1}.

We turn to the verification of \eqref{lemma:prob:cons3} and assume that assume that $I_{n}(a)=\{\hat{i}_{n}(a)\}$.  Then  by \eqref{lemma:prob:cons1} we know that
$$\lim_{\varepsilon\to 0^+}\sum_{j\in \{1,\ldots,n\}\setminus I_n(a)}p_{n,j,\varepsilon}(a)=0.$$
Since $\{p_{n,j,\varepsilon}(a)\}_{i=1}^n$ sum up one, 
$\lim_{\varepsilon\to 0^+}p_{n,\hat{i}(a),\varepsilon}(a)=1-\lim_{\varepsilon\to 0^+}\sum_{j\in \{1,\ldots,n\}\setminus I_n(a)}p_{n,j,\varepsilon}(a)=1$ follows.
This proves \eqref{lemma:prob:cons3}.

We turn our attention now to the proof of \eqref{lemma:prob:cons2}. Let assume that \eqref{lemma:prob:hyp0} and \eqref{lemma:prob:hyp1} hold. We note that given the fact  $\{p_{n,j,\varepsilon}(a)\}_{j=1}^n$ is a probability distribution, we only need to prove that  \beq \lim_{\varepsilon\to 0^+}p_{n,\hat{i}_n(a),\varepsilon}(a)=1, 
\label{lemma:prob:eq13}
\eeq
to demonstrate \eqref{lemma:prob:cons2}.  By \Cref{lemma:upperboundPsi} we know that
$$a_{\hat{i}_n(a)}=\psi_{\hat{i}_n(a)}(a)\leq \psi_{\hat{i}_n(a)-1}(a)\leq \psi_{\hat{i}_n(a)-1,\varepsilon}(a).$$
Then $a_{\hat{i}_n(a)}-\psi_{\hat{i}_n(a)-1,\varepsilon}(a)\leq 0$ for  $\varepsilon>0$. Combining this with \eqref{lemma:prob:hyp1}, \eqref{hypo:positivepart:3},  and \eqref{hypo:positivepart:5} (in the case that $a_{\hat{i}_n(a)}-\psi_{\hat{i}_n(a)-1}(a)<s_0$) we obtain
\beq \lim_{\varepsilon\to 0^+}g'_{\varepsilon}(a_{\hat{i}_n(a)}-\psi_{\hat{i}_n(a)-1,\varepsilon})(a)=0.
\label{lemma:prob:eq14}
\eeq
If $\hat{i}_{n}(a)=n$, this together with \eqref{eq:recursive:2} proves \eqref{lemma:prob:eq13}. On the other hand, if $\hat{i}_n(a)<n$, we can make use of \eqref{eq:recursive:1} instead. To this end, it is important to observe that for each $j\in\{1,\ldots,n\}$ strictly larger than $\hat{i}_n(a)$, we have, by the definition of  $\hat{i}_n(a)$, that $j\not\in I_n(a)$ and therefore arguing as in the proof of \eqref{lemma:prob:cons3} we have that 
$$\lim_{\varepsilon\to 0^+}\prod_{j=\hat{i}_n(a)+1}^n g'_{\varepsilon}(a_j-\psi_{j-1,\varepsilon}(a))=1.$$
Combining this with \eqref{lemma:prob:eq14}, and \eqref{eq:recursive:2} we obtain that \eqref{lemma:prob:eq13} holds.
\end{proof}

\begin{cor}
 If \eqref{lemma:prob:hyp0} and \eqref{lemma:prob:hyp1} hold, then for each $x\in \Omega$  all the probability is assigned to $\hat{i}_n(\Phi_n(x))$ and consequently $\nabla v_{n,\varepsilon}(x)$ converges to $\nabla \phi_{\hat{i}_n(\Phi_n(x))}(x)$ as $\varepsilon\to 0^+$.
 
\end{cor}

\begin{rem}
\label{rem:prob:HJB}
The property described in the previous corollary is particularly convenient for the value function in optimal control  as we proceed to explain. Let $v_{n}$ be a viscosity solution of an equation of the form
\beq  F(x,\nabla v(x))=0 \mbox{ in }\Omega\label{rem:prob:HJB:eq1}\eeq with $F$ continuous, and convex in the second argument. For $x\in\Omega$ let
\begin{equation*}
\begin{array}{rl} D^*v_n(x):=&\{ p\in\R^d: p=\lim_{k\to\infty}\nabla v_n(x_k),\ v_n \mbox{ differentiable at }x_k \mbox{ for all }k\in\N, \\
&\text{and } x= \lim_{k\to \infty }  x_k   \}.\label{rem:prob:HJB:eq2}
\end{array}
\end{equation*}
We recall  that  $v_n$ is semiconcave and hence $D^*v_n(x)$ is nonempty for each $x\in\Omega$, see \cite[Proposition 3.3.4(c), Theorem 3.3.6]{Cannarsa2004}. Then ,  $v_n$ satisfies \eqref{rem:prob:HJB:eq1} a.e. and we have that $F(x,p)=0$ for all $x\in\Omega$ and all $p\in D^*v_n(x)$. Here we use the continuity of $F$ and \cite[Remark 5.4.1]{Cannarsa2004}.
In particular, if for all $x\in\Omega$ and $i\in\{j\in\N: 1\leq j\leq n,\ \phi_j(x)=v_n(x)\}$ it holds
\beq \nabla \phi_{i}(x)\in D^{*} v_n(x), \label{rem:prob:HJB:eq3}\eeq
then \beq F(x,\nabla \phi_i(x))=0 \mbox{ for all } i\in\{1,\ldots,n\} \mbox{ such that }v_n(x)=\phi_i(x).
\label{rem:prob:HJB:eq5}
\eeq
In particular, assume that \eqref{rem:prob:HJB:eq3} holds with $i=\hat i_n(\Phi_n(x))$, i.e.
\beq \nabla \phi_{\hat i_n(\Phi_n(x))}(x)\in D^{*} v_n(x), \label{rem:prob:HJB:eq55}\eeq
for all $x\in \Omega$.
In this situation assigning all the probability to only one of the active functions is advantageous  for the asymptotic behavior of the $\varepsilon-$approximation. Indeed,  by \eqref{lemma:prob:cons2} and \eqref{eqkk4} we have
$$\lim_{\varepsilon\to 0^{+}}\nabla v_{n,\varepsilon}(x)=\nabla \phi_{\hat{i}_n(\Phi_n(x))}(x),\mbox{ for all }x\in\Omega,$$
which together with the continuity of $F$ and \eqref{rem:prob:HJB:eq55} implies
\beq \lim_{\varepsilon\to 0^+}F(x,\nabla v_{n,\varepsilon}(x))=F(x,\nabla \phi_{\hat{i}_n(\Phi_n(x))}(x))=0, \text{ for all }x\in\Omega. \label{rem:prob:HJB:eq4}\eeq
If \eqref{lemma:prob:hyp0} and \eqref{lemma:prob:hyp1} do not hold, the limit of $\nabla v_{n,\varepsilon}(x)$ may lie in $ D^+ v_n(x) \setminus D^* v_n(x)$  and  by Remark 5.4.1 in \cite{Cannarsa2004}[Chapter 5] we can only ensure
$$\lim_{\varepsilon\to 0^+} F(x,\nabla v_{n,\varepsilon}(x))\leq 0.$$ We shall return to this case in Remark \ref{rem:kk2} below.
\end{rem}
\begin{rem}
\label{rem:counterexample}
 It is important to observe that condition  \eqref{rem:prob:HJB:eq3} does not hold in general. If $n=2$,  then it holds trivially for arbitrary $d$, however if $n\geq 3$ it is possible to construct a counter example as follows. Let $\phi_1(x)=-x$, $\phi_{2}(x)=\exp(x)-1$, $\phi_3(x)=-x^3$ and $v_3(x)=\min_{i\in\{1,2,3\}}\phi_i(x)$ for $x\in\R$. 

We have that
$$v_{3}(x)=\left\{\begin{array}{ll}
\exp(x)-1& \mbox{ if }x\leq 0\\
-x &\mbox{ if }x\in (0,1]\\
-x^3 &\mbox{ if }x>1
\end{array}\right.$$
as is depicted in Figure \ref{fig:CounterExample:unmodified}
\begin{figure}[h!]
\begin{subfigure}[b]{0.45\textwidth}
         \centering
                  \includegraphics[width=\textwidth]{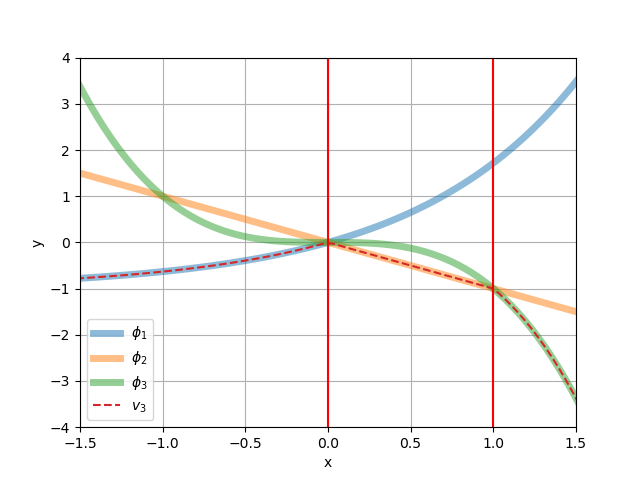}
\caption{Unmodified functions.}
\label{fig:CounterExample:unmodified}
\end{subfigure}
\begin{subfigure}[b]{0.45\textwidth}
\centering
\includegraphics[width=\textwidth]{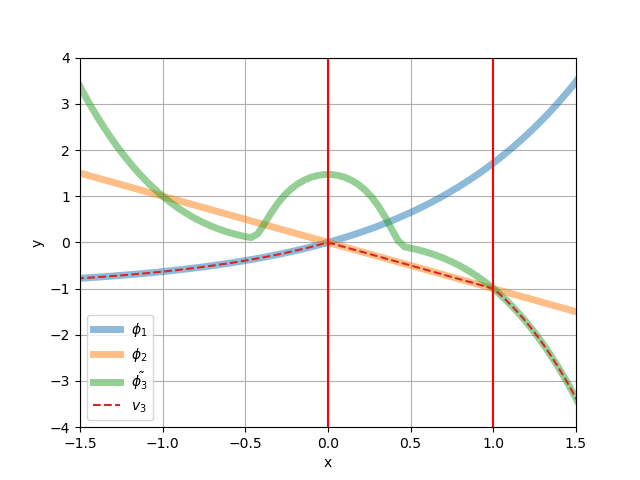}
\caption{Modified functions.}
\label{fig:CounterExample:modified}
\end{subfigure}
\caption{Illustrative example for \eqref{rem:prob:HJB:eq3}.}
\label{fig:CounterExample}
\end{figure}

At $x=0$ the three functions $\{\phi_{i}\}_{i=1}^3$ are equal to $v_{3}(0)$. Additionally, we have that 
\beq D^*v_3(0)=\{-1,1\}.
\label{rem:counterexample:eq1}
\eeq 
To prove this, we note that, for $x<0$, $v_3$ is differentiable and $v_3'(x)=\exp(x)$. Then clearly we have $\exp(0)=1$ is an element of $D^*v_3(0)$. For proving that $-1$ is an element of $D^*v_3(0)$, we note that for $x\in (0,1)$ we have that $v_{n}=-x$ and consequently $v_n'(x)=-1$ which implies that $-1\in D^*v_3(0)$. It is important to observe, that $D^*v_3(0)$ does not include any other element, since any sequence $p_{n}=v_3'(x_n)$ with $x_{n}$ converging to $0$ and $p_{n}$ converging, is such that $p_{n}=\exp(x_n)$ for all $n$ large enough or $p_{n}=-1$ for $n$ large enough.

In addition, we have that $\phi_{3}'(0)=0$ and therefore by \eqref{rem:counterexample:eq1}, it is clear that $\phi_{3}'(0)$ is not an element of $D^*(0)$.

We can remedy this issue by modifying $\phi_3$ in a neighborhood of $0$ in a way that it will not change $v_3$. For instance, let us consider $\tilde{\phi}_{3}(x)=\phi_3\left(x\right)+f(2x)$ with 
$$f(x)=\left\{ \begin{array}{ll}
\exp(-1/(1-|x|^2)) & \mbox{ if }x\in (-1,1)\\
0& \mbox{ if }|x|\geq 1
\end{array}
\right. $$
As is depicted in Figure \ref{fig:CounterExample:modified}, we have that $v_{3}(x)=\min\{\phi_{1}(x),\phi_{2}(x),\tilde{\phi}_{3}(x)\}$, but now $v_{3}(0)\neq \tilde{\phi}_{3}(0)$ and consequently the new family $\{\phi_1,\phi_2,\tilde{\phi}_3\}$ satisfies  \eqref{rem:prob:HJB:eq3}.

Using the same idea, that is, modifying the problematic $\phi_{i}$'s around the point where \eqref{rem:prob:HJB:eq3} does not hold, we show in Remark \ref{rem:FixingCounterExample} a general way to modify the family of functions $\{\phi_{i}\}_{i=1}^n$ in such a way that \eqref{rem:prob:HJB:eq3} is fulfilled everywhere.
\end{rem}

In view of Proposition \ref{lemma:prob}   with  \eqref{lemma:prob:hyp0} and \eqref{lemma:prob:hyp1} holding, the expression
\beq p_{n,i}(a):=\lim_{\varepsilon\to 0^+} p_{n,i,\varepsilon}(a)\in \{0,1\},\label{eq:kk10}\eeq
is welldefined for each $a\in\R^n$ and $i\in\{1,\ldots,n\}$.
Now for each $i=1,\ldots,n$ we define the  $\Omega^{i}$ and their corresponding approximations by $\Omega^{i,\varepsilon}$ are defined by
\beq \Omega^{i}=\{x\in\overline{\Omega}: p_{n,i}(\Phi(x))=\max_{j=1,\ldots,n}p_{n,j}(\Phi(x))\},
\label{def:acriveset}
\eeq
\beq \Omega^{i,\varepsilon}=\{x\in\overline{\Omega}: p_{n,i,\varepsilon}(\Phi(x))=\max_{j=1,\ldots,n}p_{n,j,\varepsilon}(\Phi(x))\}.
\label{def:acriveset:approx}
\eeq
We refer to the sets $\{\Omega^{i}\}_{i=1}^n$ respectively $\{\Omega^{i,\varepsilon}\}_{i=1}^n$ as gradient active sets, since if $x\in \Omega^{i,\varepsilon}$, then $\phi_{i}$ is the function with the largest contribution to $\nabla v_{n,\varepsilon}(x)$ for $\varepsilon$ sufficiently small. We also observe that $ \overline{\Omega}= \cup_{i=1}^n \Omega^i$. The following convergence result with respect to the gradient active sets can be obtained.
\begin{prop}
\label{lemma:gradientActiveSetConv}
Let $\{g_{\varepsilon}\}_{\varepsilon>0}\subset C^{1,1}(\R)$ be a family of functions satisfying \eqref{hypo:positivepart}, \eqref{lemma:prob:hyp0}, and \eqref{lemma:prob:hyp1}. Then for every $x\in\overline{\Omega}$ we have that
$$\lim_{\varepsilon\to 0^+}\chi_{\Omega_i,\varepsilon}(x)=\chi_{\Omega_i}(x), $$
where  $\chi_{\omega}$ denotes the characteristic function of $\omega\subset\R^d$.
\end{prop}

\begin{proof}[Proof of Proposition \ref{lemma:gradientActiveSetConv}]
Let $x\in\overline{\Omega}$ fixed. We start by noting that by 
\eqref{lemma:prob:cons2} we have for all $i\in\{1,\ldots,n\}$ that
$$ 
p_{n,j}(\Phi(x))=\left\{
\begin{array}{ll}
    1, & i=\hat{i} \\
   0,  & i\neq \hat{i},
\end{array}
\right.
$$
where $\hat{i}$ is the largest index for which $\phi_{i}(x)=v_{n}(x)$. From this we deduce that there exists $\varepsilon_0>0$ such that for all $\varepsilon\in (0,\varepsilon_0)$ we have 
$$ p_{n,i,\varepsilon}(\Phi(x))<\frac{1}{2}\mbox{ for }i\in \{1,\ldots,n\}\setminus\{\hat{i}\},\mbox{ and } p_{n,\hat{i},\varepsilon}(\Phi(x))>\frac{1}{2}.$$
By the definitions of $\Omega_{i,\varepsilon}$ and $\Omega_{i}$  this implies that for all $\varepsilon\in (0,\varepsilon_0)$ it holds that 
$$ \chi_{\Omega^{i,\varepsilon}}(x)=\chi_{\Omega^i}(x),$$
which implies that the pointwise convergence of $\chi_{\Omega_{i,\varepsilon}}$ as $\varepsilon\to 0^+$.
\end{proof}

\begin{rem}\label{rem:kk2}
In the following we  demonstrate the consequences of the limit behavior of $g'_{\varepsilon}$ in a neighborhood of $0$ on the  limiting behavior of the probability distribution given by $p_{n,j,\varepsilon}(a)$ as $\varepsilon \to 0^+$.

Returning to $g_{\varepsilon,M}$ and  $g_{\varepsilon,A}$ introduced in Remark \ref{rem:kk1}, both of them satisfy 
 \eqref{lemma:prob:hyp0} and \eqref{lemma:prob:hyp1}. Therefore, by means of Proposition \ref{lemma:prob} we obtain for $a\in\R^n$ and $j\in\{1,\ldots,n\}$ that
$$ p_{n,j}(a)=\left\{
\begin{array}{ll}
1 & \mbox{if }j=\hat{i}_n(a) \\
0 & \mbox{otherwise},
\end{array}
\right. $$
where we recall the notation introduced in \eqref{eq:kk10}. 

Another choice of smooth approximation for the minimum of a vector of real numbers is the Log-Sum-Exp approximation which is given by
$$ \tilde{\psi}_{n,\varepsilon}(a)=-\varepsilon\log\left(\frac{1}{n}\sum_{i=1}^n\exp\left(-\frac{a_i}{\varepsilon}\right)\right).$$
for $a\in\R^n$ and $\varepsilon>0$, which can be considered in place of $\psi_{n,\varepsilon}(a)$ defined in \eqref{defi:PsiAp}.

 The gradient of $\tilde{\psi}_n$ is given by the soft-min function:
$$\nabla \tilde{\psi}_{n,\varepsilon}(a)=\frac{1}{\sum_{i=1}^n\exp\left(-\frac{a_i}{\varepsilon}\right)}\left(\exp\left(-\frac{a_1}{\varepsilon}\right),\ldots,\exp\left(-\frac{a_n}{\varepsilon}\right)\right).$$
See \cite{GaoPavel} for a collection of results regarding this functions and its application to machine learning. We also point out that this approximation was used in the context of convexity  preserving neural networks in  \cite{CalafioreGaubertPossieri}.

Although both $\psi_{n,\varepsilon}$ and $\tilde{\psi}_{n,\varepsilon}$ converge  to $\psi_{n}$ uniformly as $\varepsilon \to 0^+$,  the limiting  behaviors as $\varepsilon \to 0^+$ of the derivatives of these two approximation are different. This has an impact on the limiting properties of the approximation of $\nabla v_n$ provided  by $\tilde v_{n,\epsilon}=\tilde\psi_{n,\varepsilon}\circ \Phi_n$  in comparison to $v_{n,\varepsilon}$. To see this, we note that setting $I(a)=\argmin\{a_i:i\in\{1,\ldots,n\}\}$ for $a\in\R^n$ we have for $j\in\{1,\ldots,n\}$ that
$$\lim_{\varepsilon\to 0^+}\frac{\partial \tilde{\psi}_{n,\varepsilon}}{\partial a_j}(a)=\left\{
\begin{array}{ll}
\frac{1}{|I(a)|} & \mbox{if }j\in I(a),\\
0 & \mbox{if }j\notin I(a).
\end{array}\right.
 $$
From this observation we see that $\nabla \tilde{v}_{n,\varepsilon}$ satisfies for each $x\in\Omega$ that
$$ \lim_{\varepsilon\to 0^+}\nabla \tilde{v}_{n,\varepsilon}(x)=\frac{1}{|I(\Phi_n(x))|}\sum_{i\in I(\Phi_n(x))}\nabla \phi_i(x),$$
that is, $\nabla\tilde{v}_{n,\varepsilon}$ converges point-wise to the average of the gradients of the active functions. In line with Remark \ref{rem:prob:HJB}, if $v_n$ is a viscosity solution of an  equation of the form of \eqref{rem:prob:HJB:eq1}, we can only ensure that \beq \lim_{\varepsilon\to 0^+}F(x,\nabla \tilde{v}_{n,\varepsilon}(x))\leq 0,\eeq
compared to
\beq \lim_{\varepsilon\to 0^+}F(x,\nabla v_{n,\varepsilon}(x))= 0,\eeq
if \eqref{lemma:prob:hyp0} and \eqref{lemma:prob:hyp1} hold.
Consequently in the context of viscosity solutions,  $\psi_{n,\varepsilon}$ provides a better representation of $v_n$ at points of discontinuity   than $\tilde{\psi}_{n,\varepsilon}$.
\end{rem}

\begin{rem}
\label{rem:FixingCounterExample}
Regarding Remark \ref{rem:prob:HJB},  if condition \eqref{rem:prob:HJB:eq3} does not hold, we now explain how to modify the functions $\{\phi_{i}\}_{i=1}^n$ in a manner such  that $v_n$ does not change and the new family satisfies \eqref{rem:prob:HJB:eq3}. For this purpose let us first fix some $i\in\{1,\ldots,n\}$  and define
$$A_i=\{x\in\overline{\Omega}: \nabla \phi_{i}(x)\notin D^*v_n(x) \}.$$
We argue that $A_{i}$ is open in the relative topology of $\overline{\Omega}$. Indeed, let $\bar{x}\in\overline{\Omega}$ be such that $\nabla \phi_{i}(\bar{x})\notin D^*v_n(\bar{x})$, then there exists $\delta>0$ such that for all $x\in B(\bar{x},\delta)\cap \overline{\Omega}$ we have $\nabla \phi_{i}(x)\notin D^*v_n(x)$. To prove this claim, let us proceed by contradiction. If the claim does not hold, then there exists $x_{m}\in\overline{\Omega}$ approaching $\bar{x}$ as $m\to\infty$ such that $\nabla \phi_i(x_m)\in D^* v_n(x_m)$. By the definition of $D^* v(x_m)$, for each $m\in\N$ there exists a sequence $x_{k,m}$ such that $x_{k,m}\to x_m$ as $k\to\infty$, $v_n$ is differentiable at $x_{k,m}$ and $\nabla v_n(x_{k,m})\to \nabla \phi_i(x_m)$ as $k\to\infty$. From this we deduce that for each $m\in\N$ there exists $k(m)$ such that
$$|\nabla v_n(x_{k(m),m})-\nabla \phi_i(x_m)|+|x_{k,m}-x_m|<\frac{1}{m}. $$ 
Then by setting $\tilde{x}_m=x_{k(m),m}$ we have that
$$|\nabla v_n(\tilde{x}_{m})-\nabla \phi_i(\bar{x})|+|x_{k,m}-\bar{x}|<\frac{1}{m}+|\nabla \phi_i(\bar{x})-\nabla \phi_i(x_m)|+|x_m-\bar{x}|.$$
Thus, by the continuity of $\nabla \phi_i$ and the convergence of $x_{m}$ to $\bar{x}$ as $m\to\infty$, we deduce that
$$\lim_{m\to\infty} \nabla v_n(\tilde{x}_{m})=\nabla \phi_i(\bar{x}) \mbox{ and }\lim_{m\to\infty}\tilde{x}_{m}=\bar{x}.$$
This implies that $\nabla \phi_i(\bar{x})\in D^* v_n(\bar{x})$, which is a contradiction and consequently $A_i$ is open in the relative topolgy of $\overline{\Omega}$. 

In particular, this implies that for each $x\in  A_i$ there exists $j\in \{1,\ldots,n\}\setminus\{i\}$ with $v_{n}(x)=\phi_j(x)$. Indeed, if the aforementioned assertion does not hold, then there exists $x\in A_i$ such that $v_{n}(x)=\phi_{i}(x)$ and $v_{n}(x)<\phi_j(x)$ for all $j\in\{1,\ldots,n\}\setminus\{i\}$. By the continuity of $\phi_i$, this implies that there exists $\delta>0$ such that for all $y\in B(x,\delta)\cap\Omega$ we have that $\phi_{i}(y)<\phi_j(y)$ for $j\in\{1,\ldots,n\}\setminus\{i\}$ and consequently $v_{n}=\phi_i$ in $B(x,\delta)\cap\Omega$. From this we deduce that $v_n$ is $C^1$ at $x$ and that $\nabla v_n(x)=\nabla \phi_i(x)$. However, this contradict the fact that $x\in A_i$. 

Using that for all $x\in A_i$ there exists $j\in\{1,\ldots,n\}\setminus\{i\}$ satisfying $v_{n}(x)=\phi_j(x)$ , we have that adding a positive number to $\phi_i(x)$ in a neighborhood of $x$ will not change $v_n(x)$ since $\phi_j(x)$ remains unchanged.

Since $A_{i}$ is open and due to the Lindelof property of $\R^d$, there exists an open sub-covering $\{B(x_j,\delta_j)\cap\overline{\Omega}\}_{j=1}^\infty$ of $A_i$ such that
$$ A_i=\bigcup_{j=1}^\infty  B(x_j,\delta_j)\cap\overline{\Omega}.$$
Since $A_{i}$ is bounded we can assume that $\{\delta_j\}_{j=1}^\infty$ is bounded. Let us consider a function $\nu:\R^d\mapsto\R$ such that $\nu\in C_{c}^{\infty}(\R^d)$ and $\nu(x)>0$ in $B(0,1)$ and $\nu(x)=0$ in $\R^d\setminus B(0,\delta)$. We define $\tilde{\phi}_{i}$ by
$$\tilde{\phi}_{i}(x)=\sum_{j=1}^{\infty}\frac{\delta_j^2}{2^j}\nu\left(\frac{x-x_j}{\delta_j}\right) +\phi_i(x)\mbox{ for }x\in\overline{\Omega}.$$
We clearly have that $\tilde{\phi}_{i}\in C^2(\overline{\Omega})$, $\tilde{\phi}_{i}(x)>\phi_i(x)$ for all $x\in A_i$ and $\tilde{\phi}_i(x)=\phi_i(x)$ for all $x\in\overline{\Omega}\setminus A_i$. Then replacing $\tilde{\phi}_i$ by $\phi_i$ does not affect the definition of $v_n$ and we have that $\nabla \tilde{\phi}_{i}(x)\in D^* v_n(x)$ for every $x\in\overline{\Omega}$ where $\tilde{\phi}_i(x)=v_n(x)$.
Now we iterate from $i=1$ to $i=n$.  If $A_{i}$ is empty we do not need to do anything, on the other hand, if $A_{i}\neq \emptyset$, we modify $\phi_i$ as above. In this manner we obtain that at the end of the $i$-th iteration we have 
$$v_n(x)=\min \left\{\min_{k\in\{1,\ldots,i\}}\tilde{\phi}_k(x),\min_{k\in\{i+1,\ldots,n\}}\phi_k(x)\right\}\text{ for all }x\in\overline{\Omega}$$ $$\text{ and } \{x\in\overline{\Omega}: \nabla \tilde{\phi}_{j}(x)\notin D^*v_n(x)\mbox{ and }\tilde{\phi}_j(x)=v_n(x) \} =\emptyset$$
 for $j\in\{1,\ldots,i\}$.
In this  manner  we end up with a modified family $\{\tilde{\phi}_{i}\}_{i=1}^n$ such that 
$$ v_n(x)=\min_{i\in\{1,\ldots,n\}}\tilde{\phi}_{i}(x), \text{ and } \{x\in\overline{\Omega}: \nabla \tilde{\phi}_{i}(x)\notin D^*v_n(x)\mbox{ and }\tilde{\phi}_i(x)=v_n(x) \}=\emptyset.$$
\end{rem}

\section{Approximation}
\label{sec:approximation}

Our next goal is to analyze the universality of the parametrization $\varv_{n,\varepsilon}$, which was introduced at the end of  Section \ref{sec:parametrization}, for semiconcave functions. We also address approximation properties involving the active sets of semiconcave functions. For this purposes we introduce a sequence  of settings $\{(\Theta_{m},\norm{\cdot}_{m}),\xi_{m}\}_{m\in\N}$, which will be required to satisfy the following property:

\begin{hypo}
\label{hypo:approx}
For all $\phi\in C^2(\overline{\Omega})$ and $m\in\N$ there exist $\theta_{m}\in\Theta_m$, such that
\beq  \lim_{m\mapsto\infty }\xi_{m}(\theta_m)=\phi\mbox{ in }C^2(\overline{\Omega}).
\label{hypo:approx:eq1}
\eeq
\end{hypo}
In analogy to the notation introduced in Section \ref{sec:parametrization} we define for the  sequence of settings $\{(\Theta_{m},\norm{\cdot}_{m}),\xi_{m}\}_{m\in\N}$,  the mappings $$\Xi_{m}:(\Theta_m)^n\mapsto (C^2(\Omega))^n \text{ by } \Xi_m(\theta_1,\ldots,\theta_m)=(\xi_m(\theta_1),\ldots,\xi_m(\theta_n)) \text{ and}$$
  $$\varv_{m,n,\varepsilon}:\Theta_m^n\mapsto C^2(\overline{\Omega}) \text{ by } \varv_{m,n,\varepsilon}(\theta)=\psi_{n,\varepsilon}(\Xi_m(\theta_1,\ldots,\theta_m))$$ for   $\theta=(\theta_{1},\ldots,\theta_{n})\in (\Theta_m)^n$.
We start with a technical result.

\begin{prop}\label{prop:interpolation}
There exists a constant $K>0$ depending on $\Omega$ such that for all $v\in Lip (\bar \Omega)$ with $\nabla v\in BV(\Omega;\R^d)$ and
   $p\in [1,\infty)$ the following inequality holds:
     \begin{equation}\label{prop:interpolation:cons1}
    \norm{\nabla v}_{L^p(\Omega;\R^d)} \leq \left\{ \begin{array}{ll}
 \displaystyle  \norm{\nabla v}_{L^{\infty}(\Omega;\R^d)}^{\frac{p-2}{p}}\left(K\norm{v}_{C(\overline{\Omega})}\norm{\nabla v}_{BV(\Omega;\R^d)}\right)^{\frac{1}{p}} & \text{ if }p\in [2,\infty), \\
\ecart\displaystyle |\Omega|^{\frac{2-p}{2p}}\left(K\norm{v}_{C(\overline{\Omega})}
\norm{\nabla v}_{BV(\Omega;\R^d)}\right)^{\frac{1}{2}} & \text{ if }p\in {{[1,2)}}.
     \end{array}
\right.     
\end{equation}

\end{prop}
\begin{proof}
Since $\nabla v$ is a function of bounded variation we can employ Theorem 5.6 in \cite[Chapter 5]{Evans2} to assert that:
\beq\forall\phi\in C^1(\Omega):\ \int_{\Omega}\nabla v\nabla \phi dx=-\int_{\Omega}\phi\, d\dive(\nabla v) +\int_{\partial\Omega} \phi \nabla v\cdot \hat{\nu} d\mathcal{H}_{d-1},\label{prop:interpolation:eq1}\eeq
where $\mathcal{H}_{d-1}$ is the $(d-1)$-Hausdorff measure, $\hat{\nu}$ is the outward normal to $\partial\Omega$, and the $d\dive(\nabla v)$ is a signed and finite Radon measure in $\Omega$. To see that this holds, we have from Theorem 5.6 in \cite[Chapter 5]{Evans2} and the fact that $\frac{\partial v}{\partial x_i}\in BV(U)$ that 

    $$\int_{\Omega}\nabla v\nabla \phi dx =\sum_{i=1}^{d}\int_{\Omega}\frac{\partial v}{\partial x_i}\dive\left( e_i\phi \right)dx=\sum_{i=1}^d\left( -\int_{\Omega} \phi e_i d\nabla \frac{\partial v}{\partial x_i}+\int_{\partial\Omega} \phi \frac{\partial v}{\partial x_i} e_{i} \cdot\nu d\mathcal{H}_{d-1}\right) $$
    $$ =-\int_{\Omega}\phi\, d\dive(\nabla v) +\int_{\partial\Omega} \phi \nabla v\cdot \hat{\nu} d\mathcal{H}_{d-1}$$
where  for $\in\{1,\ldots,d\}$ $e_i$ corresponds to the $i-th$ canonical vector of $\R^d$ and the expression $d\nabla \frac{\partial v}{\partial x_i}$ stands for the vector valued radon measure.

Again by Theorem 5.6 in \cite[Chapter 5]{Evans2}, there exists a constant $C>0$ such that 
\beq \forall\phi\in C^1(\overline{\Omega}):\ \left|\int_{\partial\Omega} \phi \nabla v\cdot \hat{\nu} d\mathcal{H}_{d-1}\right|\leq C\norm{\nabla v}_{BV(\Omega;\R^d)}\norm{\phi}_{C(\overline{\Omega})}.\label{prop:interpolation:eq2}\eeq
Here we use that the boundary trace of BV functions is a bounded linear operator from $BV(\Omega)$ to $L^1(\partial \Omega; \mathcal{H}_{d-1})$.
Combining \eqref{prop:interpolation:eq1}, \eqref{prop:interpolation:eq2} and the fact that 
$$\forall\phi\in C^1(\Omega):\ \left|\int_{\Omega}\phi  d\dive \nabla v\right|\leq d|\dive \nabla v|(\Omega)\norm{\phi}_{C(\overline{\Omega})}\leq \norm{\nabla v}_{BV(\Omega;\R^d)}\norm{\phi}_{C(\overline{\Omega})},$$ we have there exists $K>0$ independent of $v$ such that
\beq\forall\phi\in C^1(\Omega):\ \left|\int_{\Omega}\nabla v\nabla \phi dx\right|\leq K \norm{\nabla v}_{BV(\Omega;\R^d)}\norm{\phi}_{C(\overline{\Omega})}. \label{prop:interpolation:eq3}\eeq
Additionally, due to the fact that $v\in C(\overline{\Omega})\cap W^{1,\infty}(\Omega)$, there exists a sequence $v_{n}\in C^{\infty}(\overline{\Omega})$ satisfying:
$$ \lim_{n\to\infty}\norm{v_{n}-v}_{W^{1,2}(\Omega)}=0,\quad\norm{v_n-v}_{C(\overline{\Omega})}=0.$$
Using \eqref{prop:interpolation:eq3} with $\phi=v_{n}$ we deduce that 
\beq \norm{\nabla v}^2_{L^2(\Omega;\R^d)}\leq K \norm{\nabla v}_{BV(\Omega;\R^d)}\norm{v}_{C(\overline{\Omega})}. \label{prop:interpolation:eq4}\eeq
This proves \eqref{prop:interpolation:cons1} in the case $p=2$. If $p\in (1,2)$, the Hölder inequality implies that 
$$\norm{\nabla v}^{p}_{L^p(\Omega;\R^d)}\leq |\Omega|^{\frac{2-p}{2}}\left(K \norm{\nabla v}_{BV(\Omega;\R^d)}\norm{v}_{C(\overline{\Omega})}\right)^{\frac{p}{2}} $$
which proves \eqref{prop:interpolation:cons1} in this case.

For $p>2$, due to the fact that $v\in Lip(\overline{\Omega})$ we have that 
$$\norm{\nabla v}_{L^{p}(\Omega;\R^d)}\leq \norm{\nabla v}_{L^{\infty}(\Omega;\R^d)}^{\frac{p-2}{p}}\norm{\nabla v}_{L^2(\Omega)}^{\frac{2}{p}}.$$
Combining the above inequality with \eqref{prop:interpolation:eq4} we obtain that \eqref{prop:interpolation:cons1} holds.
\end{proof}

\begin{rem}
We observe that \Cref{prop:interpolation} allows to transfer information from a function to its gradient. For instance, for $v=f-f_n$, with $$\max\left\{\norm{f}_{ W^{1,\infty}(\overline{\Omega})} +\norm{\nabla f}_{BV(\Omega;\R^d)},\sup_{n\in\N}\norm{f_n}_{ W^{1,\infty}(\overline{\Omega})}+\norm{\nabla f}_{BV(\Omega;\R^d)}\right\}\leq R$$ the previous proposition allows us to measure the convergence of $\nabla f_n$ to $\nabla f$ in terms of the $C(\overline{\Omega})$ distance between $f_n$ and $f$.
\end{rem}

We shall apply  Proposition \ref{prop:interpolation} to the case of semiconcave function. We recall that according to Theorem 2.3.1 in \cite[Section 3, Chapter 2]{Cannarsa2004}, semiconcave functions belong to $BV(\Omega)$.  This results from the fact that convex functions are of bounded variation. We cannot use this fact directly, since it does not provide a bound on the $BV(\Omega;\R^d)$ norm. For this reason, in the following theorem we first give an estimate of this norm depending on the semiconcavity constant, on $\Omega$, and the $C(\overline{\Omega})$ norm, which we combine with \eqref{prop:interpolation:cons1}.

\begin{lemma}
\label{lemma:BVnormSC}
There exists a constant $K>0$ depending on $\Omega$ such that for all $v\in Lip (\bar \Omega)$ which are $C-$semiconcave with $C>0$ the estimate
\beq\norm{\nabla v }_{BV(\Omega;\R^d)}\leq K\left(\norm{\nabla v}_{L^\infty(\Omega;\R^d)}+C+\norm{v}_{C(\overline{\Omega})}\right)\label{lemma:BVnormSC:cons1} \eeq
holds.
\end{lemma}
\begin{proof}
The proof consists of several steps. First  $v$ will be extended to a function $\hat{v}$ on  $\R^d$ with the same semiconcavity constant. Afterwards, we shall prove that $\nabla\hat{v}$ is a function in $BV(\Omega)$ thanks to the fact that $x\in\R^d\mapsto \frac{C}{2}|x|^2-\hat{v}(x)$ is convex. Additionally, delving into the proof of Theorems 6.8 and Theorem 1.39 in \cite{Evans2}, we will provide a bound on the $BV(\Omega;\R^d)$ norm of $\nabla\hat{v}$ depending on $\Omega$, $\norm{\nabla v}_{W^{1,p}(\overline{\Omega})}$, and $\norm{v}_{C(\overline{\Omega})}$. Thus \eqref{lemma:BVnormSC:cons1} will follow from the fact that $\hat{v}$ is an extension of $v$.

The first step is to extend $v$ to $\R^d$ preserving its semiconcavity constant. Following Proposition 3.1 in \cite{Albano} we extend $v$ from $\overline{\Omega}$ to $\R^d$ by:
$$\forall y\in\R^d:\ \hat{v}(y)=\inf_{z\in\Omega,\ p\in D^+v(z)}\left\{v(z)+p^\top\cdot(y-z)+\frac{C}{2}|y-z|^2\right\}.$$

Let us consider a set $U\subset\R^d$ bounded and open, we shall prove that $\hat{v}$ is Lipschitz and semiconcave.

It is clear that $\hat{v}$ is the pointiwise minimum over the following family of functions 
$$z\in\Omega,p\in D^+v(z):\ y\in\R^d\mapsto v_{z,p}(y)= v(z)+p^{\top}(y-z)+\frac{C}{2}|y-z|^2.$$ In addition, for all $z\in\Omega,p\in D^+v(z)$ the function $v_{z,p}$ is $(|p|+C\sup_{y\in U}d(z,y))$-Lipschitz in $U$ and $C$-semiconcave. We note that by the same argument as in the proof of \Cref{theo:SemiConcaveRepresentation} we have that 
$$\forall z\in \Omega,\ p\in D^+v(z):\ |p|\leq \norm{\nabla v}_{L^\infty(\Omega)}.$$
Therefore by 
Proposition 1.32 in \cite{Weaver} and \Cref{theo:SemiConcaveRepresentation} we have that $\hat{v}$ restricted to $U$ is Lipschitz with constant
$$L(U):=\norm{\nabla v}_{L^\infty(\Omega;\R^d)}+C\sup_{y\in U}dist(\Omega,y),$$ where for $y\in\R^d$ the expression $dist(\Omega,y)$ is the distance from $\Omega$ to $y$, and $C-$semiconcave. Since $U$ is arbitrary, we get that $\hat{v}$ is locally Lipschitz in $\R^d$ and $C-$semiconcave in $\R^d$.

It is easy to see that $\hat{v}$ and $v$ coincide in $\Omega$ and therefore by the continuity of both functions we have that $\hat{v}=v$ in $\partial\Omega$. Hence, $\hat{v}$ is an extension of $v$ with the same semiconcavity constant.

We now estimate $\norm{\hat{v}}_{C(\overline{U})}$. For all $y\in U$ we have that 
$$|\hat{v}(y)|\leq |v(\mathcal{P}_{\overline{\Omega}}(y))|+|v(\mathcal{P}_{\overline{\Omega}}(y))-\hat{v}(y)|.$$
where $\mathcal{P}_{\overline{\Omega}}$ is the projection of onto $\overline{\Omega}$ which is well defined since $\Omega$ is convex and bounded. Since $v$ and $\hat{v}$ coincide in $\overline{\Omega}$ and $\hat{v}$ is $L(U)-$Lipschitz, we have that 
$$\forall y\in U:\ |\hat{v}(y)|\leq |v(\mathcal{P}_{\overline{\Omega}}(y))|+L(U)|y-\mathcal{P}(y)|\leq \norm{v}_{C(\overline{\Omega})}+L(U)dist(y,\overline{\Omega}).$$
Consequently, we get that 
\beq \norm{\hat{v}}_{C(\overline{U})}\leq \norm{v}_{C(\overline{\Omega})}+L(U)\sup_{y\in U}dist(y,\overline{\Omega}).  \label{lemma:BVnormSC:proof:eq1}\eeq

We now  estimate the $BV(\Omega;\R^d)$ norm of $\hat{v}$. To this end, and since $\Omega$ is bounded, if $D$ is the diameter of $\Omega$, then there exists $\bar{x}\in\R^d$ such that $\overline{\Omega}$ is strictly contained in  $B(\bar{x},2D)$. Let us consider a function $\kappa\in C^{\infty}_{c}(B(\bar{x},2D))$ such that 
$$\forall x\in\R^d:\ \kappa (x)\in [0,1],\text{ and } \kappa(x)=1\Leftrightarrow x\in \overline{\Omega}.$$

Let us set $g(x)=\frac{C}{2}|x|^2-\hat{v}$ for $x\in\R^d$. Since $\hat{v}$ is $C$-semiconcave, we have that $g$ is convex and by \cite[Theorem 6.8]{Evans2} 
we deduce that $\nabla g$ is an element of $BV_{loc}(\R^d;\R^d)$. Since $\nabla g(x)=Cx-\nabla \hat{v}(x)$ for almost all $x\in\R^d$, we infer that $\hat{v}\in BV_{loc}(\R^d;\R^d)$.

To estimate the total variation norm of $\frac{\partial^2 \hat{v}}{\partial x_i\partial x_j} $ as a Radon measure in $\Omega$ we follow the spirit of the proof of Theorem 1.39 in \cite{Evans2}. Let us consider $\phi\in C^\infty_{c}(\Omega)$ and $p\in\R^d$ a vector with $|p|=1$, and assume, for the moment,  that $\hat{v}\in C^2(\R^d)$. Set $\tilde{\phi}=\norm{\phi}_{C(\overline{\Omega})}\kappa-\phi$. By the semiconcavity of $v$ we have that
$C I-\nabla^2 v(x)$ is semidefinite and thus
$$\int_{\R^d} (C-p^\top \nabla^2 v(x)p)\tilde{\phi}(x) dx\geq 0 $$
and since $\tilde{\phi}\ge\norm{\phi}_{C(\overline{\Omega})}\kappa-\phi\geq 0$ on $\Omega$ we get 
$$ \int_{\R^d} (C-p^\top \nabla^2 v(x)p)\phi(x) dx\leq \norm{\phi}_{C(\overline{\Omega})}\int_{\R^d}(C-p^\top \nabla^2 v(x)p)\kappa (x)dx.$$
Integrating by part the last term in the right-hand side of the above inequality we obtain 
$$ \int_{\R^d} (C-p^\top \nabla^2 v(x)p)\phi(x) dx\leq \norm{\phi}_{C(\overline{\Omega})}|B(\bar{x},2D)|(C+\norm{\nabla^2\kappa}_{C(\R^{d\times d};\R^d)}\norm{\hat{v}}_{C(\bar{B}(\bar{x},2D))}).$$
Then we have 
$$
\begin{array}{l}
\displaystyle 
-\int_{\Omega} p^\top \nabla^2 v(x)p\phi(x)dx=C\int_{\R^d}\phi(x)dx+\int_{\R^d} (C-p^\top \nabla^2 v(x)p)\phi(x) dx\\
\ecart \displaystyle \leq \norm{\phi}_{C(\overline{\Omega})}\left(C(|B(\bar{x},2D)|+|\Omega|)+|B(\bar{x},2D)|\norm{\nabla^2\kappa}_{C(\R^{d\times d};\R^d)}\norm{\hat{v}}_{C(\bar{B}(\bar{x},2D))}\right).
\end{array} $$
Since $\phi$ is arbitrary and grouping all the terms depending on $\Omega$, we deduce then that there exists $K>0$ depending on $\Omega$ such that:
\beq\forall\phi \in C_{c}^2(\Omega):\ \left|\int_{\Omega} p^\top \nabla^2 v(x)p\phi(x)dx\right|\leq K\norm{\phi}_{C(\overline{\Omega})}\left(C+\norm{\hat{v}}_{C(\bar{B}(\bar{x},2D))}\right).\label{lemma:BVnormSC:proof:eq2} \eeq
Taking $p=e_{i}$ for $i\in\{1,\ldots,d\}$ being the $i-$th canonical vector of $\R^d$, then 
\beq\forall\phi \in C_{c}^2(\Omega):\ \left|\int_{\Omega} \frac{\partial^2\hat{v}}{\partial x_i^2}\phi(x)dx\right|\leq K\norm{\phi}_{C(\overline{\Omega})}\left(C+\norm{\hat{v}}_{C(\bar{B}(\bar{x},2D))}\right). \label{lemma:BVnormSC:proof:eq3}\eeq
Additionally, taking $p=\frac{1}{\sqrt{2}}(e_{i}+e_j)$ for $i, j\in\{1,\ldots,d\}$ with $i\neq j$  we note that 
$$p^\top \nabla^2 v(x)p=\frac{1}{2}\left(\frac{\partial^2\hat{v} }{\partial x_{i}^2}(x)+\frac{\partial^2\hat{v} }{\partial x_{j}^2}(x)+2\frac{\partial^2\hat{v} }{\partial x_{i}\partial x_j}(x)\right) $$
and thus 
$$\frac{\partial^2\hat{v} }{\partial x_{i}\partial x_j}(x)=p^\top \nabla^2 v(x)p-\frac{1}{2}\left(\frac{\partial^2\hat{v} }{\partial x_{i}^2}(x)+\frac{\partial^2\hat{v} }{\partial x_{j}^2}(x)\right).$$
Combining this with \eqref{lemma:BVnormSC:proof:eq2} we get for all $i,j \in \{1,\dots,d\}$
\beq\forall\phi \in C_{c}^2(\Omega):\ \left|\int_{\Omega} \frac{\partial^2\hat{v}}{\partial x_i\partial x_j}\phi(x)dx\right|\leq 2K\norm{\phi}_{C(\overline{\Omega})}\left(C+\norm{\hat{v}}_{C(\bar{B}(\bar{x},2D))}\right). 
\label{lemma:BVnormSC:proof:eq4}
\eeq
We note that estimates \eqref{lemma:BVnormSC:proof:eq3}and \eqref{lemma:BVnormSC:proof:eq4} imply the existence of a constant independent of $v$, which we still denote by $K>0$, such that 
$$\norm{\nabla \hat{v}}_{BV(\Omega;\R^d)}\leq K \left(C+\norm{\hat{v}}_{C(\bar{B}(\bar{x},2D))}\right). $$
Combining this with \eqref{lemma:BVnormSC:proof:eq1} and the fact that $v=\hat{v}$ in $\Omega$, we arrive at \eqref{lemma:BVnormSC:cons1}. We recall that we assumed that $\hat{v}\in C^2(\R^d)$. The general case follows from a density  argument. That is, for $\varepsilon>0$ we denote by $\hat{v}_{\varepsilon}$ an smooth mollification of $\hat{v}$. Since $\hat{v}$ is defined globally, we have that $\hat{v}_{\varepsilon}\in C^{\infty}(\R^d)$ as by classical properties of a mollification we have that (see for instance Theorem 4.1 in \cite{Evans2})
\beq\lim_{\varepsilon\to 0^+}\hat{v}_{\varepsilon}=\hat{v}\text{ in }C(\overline{U}),\ \norm{\nabla \hat{v}_{\varepsilon}}_{L^\infty(U;\R^d)}\leq \norm{\nabla \hat{v}}_{L^\infty(U;\R^d)}, 
\label{lemma:BVnormSC:proof:eq5}
\eeq
and by Proposition 1.3.3 in \cite{Cannarsa2004} we have that $\hat{v}_{\varepsilon}$ is $C$ semiconcave. Furthermore by Theorem 5.2 in \cite{Evans2} we have that 
\beq\norm{\nabla v}_{BV(\Omega;\R^d)}\leq \liminf_{\varepsilon\to 0^+} \norm{\nabla v_\varepsilon}_{BV(\Omega;\R^d)}.
\label{lemma:BVnormSC:proof:eq6}
\eeq
Thus we get that \eqref{lemma:BVnormSC:cons1} is satisfied by $\hat{v}_{\varepsilon}$ for a constant $K>0$ independent of $\varepsilon$. The general case then follows from this, \eqref{lemma:BVnormSC:proof:eq5} and \eqref{lemma:BVnormSC:proof:eq6}.
\end{proof}

Combining  \Cref{prop:interpolation} and \Cref{lemma:BVnormSC}, we obtain \Cref{Theo:errorbounds} below. This  is an approximation   result for semiconcave functions for  the architecture proposed in this work. The  emphasis lies on the fact  that an  approximation of the functions  $\{\phi_i\}_{i=1}^{n}$ in the $C(\overline{\Omega})$ norm, induces information on the approximation  of the gradient of a  semiconcave function. 

\begin{theo}
    \label{Theo:errorbounds}
    For $i=1,\ldots,n$ let $\{\phi_{i,m}\}_{m=1}^\infty  \subset C^2(\overline{\Omega})$ be  sequences  converging to $\phi_i$ in $C(\overline{\Omega})$. Additionally,  assume that there exist $L>0$ and $C>0$ satisfying
\beq \limsup_{m\to\infty}L_m\leq L
\label{Theo:errorbounds:hyp1}
\eeq
and
\beq
\limsup_{m\to\infty}C_{m}\leq C,
\label{Theo:errorbounds:hyp2}
 \eeq
where

$$L_{m}=\sup_{i\in\{1,\ldots,n\}}
\norm{\nabla \phi_{i,m}}_{C(\overline{\Omega};\R^d)} \text{ and }C_{m}=\sup_{i\in\{1,\ldots,n\}}\norm{\nabla^2 \phi_{i,m}}_{C(\overline{\Omega};\R^{d\times d})}.$$
For $m\in\N$ and $\varepsilon\geq 0$ let us set 
    $$ v_n=\min_{i=1,\ldots,n}\phi_i,\ v_{n,m,\varepsilon}=\psi_{n,\varepsilon}\circ\Phi_{n,m}$$
where $\Phi_{n,m}\in C^2(\overline{\Omega};\R^n)$ is defined by $x\in\overline{\Omega}\mapsto\Phi_{n,m}(x)=(\phi_{1,m}(x),\ldots,\phi_{n,m}(x))$. Then
    \beq \norm{v_n-v_{n,m,\varepsilon}}_{C(\overline{\Omega})}\leq \sup_{i=1,\ldots,n}\norm{\phi_i-\phi_{i,m}}_{C(\overline{\Omega})}+(n-1)\varepsilon,
    \label{Theo:errorbounds:eq1}
    \eeq
and, if $p\geq 2$:
\beq
\begin{array}{l}
\displaystyle\norm{\nabla v_{n,m,\varepsilon}-\nabla v_{n}}_{W^{1,p}(\Omega)}\leq \\
\ecart \displaystyle L^{\frac{p-2}{p}}\left(K\left(\sup_{i=1,\ldots,n}\norm{\phi_i-\phi_{i,m}}_{C(\overline{\Omega})}+(n-1)\varepsilon\right)
 \left(C+L+\sup_{i=1,\ldots,n}\norm{\phi_i-\phi_{i,m}}_{C(\overline{\Omega})}+(n-1)\varepsilon\right)\right)^{\frac{1}{p}}
\end{array}
, \label{Theo:errorbounds:eq2:2:1}\eeq    
if $p\in [1,2)$:
\beq
\begin{array}{l}
\displaystyle\norm{\nabla v_{n,m,\varepsilon}-\nabla v_{n}}_{W^{1,p}(\Omega)}
\leq\\
\ecart\displaystyle  |\Omega|^{\frac{2-p}{2p}}\left(K\left(\sup_{i=1,\ldots,n}\norm{\phi_i-\phi_{i,m}}_{C(\overline{\Omega})}+(n-1)\varepsilon\right)\left(C+L+\sup_{i=1,\ldots,n}\norm{\phi_i-\phi_{i,m}}_{C(\overline{\Omega})}+(n-1)\varepsilon\right)\right)^{\frac{1}{2}} 
\end{array}
\label{Theo:errorbounds:eq2:2:2}
\eeq

    where $K$  is a constant depending only on $\Omega$. Further $v_n$ is $C-$semiconcave and $L-Lipschitz$.
\end{theo}

\begin{proof}
To prove this result we recall that for all $\varepsilon>0$, the function $\psi_{n,\varepsilon}$ is a 1-Lipschitz for the topology induced by the norm $\norm{\cdot}_{\infty}$. Therefore, we have that
\begin{equation*}
\begin{array}{l} 
\displaystyle|v_{n}(x)-v_{n,m,\varepsilon}(x)|=|\psi_{n}(\Phi(x))-\psi_{n,\varepsilon}(\Phi_m(x))(x)|\leq \\
\ecart\displaystyle
|\psi_{n}(\Phi(x))-\psi_{n}(\Phi_m(x))(x)|+|\psi_{n,\varepsilon}(\Phi_m(x))-\psi_{n}(\Phi_m(x))(x)| \\
\ecart\displaystyle
\le \sup_{i\in\{1,\ldots,n\}}|\phi_{i}(x)-\phi_{i,m}(x)|+(n-1)\varepsilon, 
\end{array}\end{equation*}
where we have used \eqref{prop:propertiesApprox:cons3}. This proves \eqref{Theo:errorbounds:eq1}.

To prove \eqref{Theo:errorbounds:eq2:2:1}-\eqref{Theo:errorbounds:eq2:2:2} we note that due to \eqref{prop:propertiesApprox:cons3}, \eqref{hypo:positivepart:2}, and \eqref{eqkk4} the functions $v_{n,m,\varepsilon}$ are $L_m$-Lipschitz. Therefore we have
$$|v_{n,m,\varepsilon}(x)-v_{n,m,\varepsilon}(y)|\leq L_m|x-y|\mbox{ for all }x,y\in\overline{\Omega}.$$
Using this, \eqref{Theo:errorbounds:eq1}, \eqref{Theo:errorbounds:hyp1}, and the assumption that $\phi_{i,m} \to \phi_i$ in $C(\bar \Omega)$  we obtain
$$|v_{n}(x)-v_{n}(y)|\leq L|x-y|\mbox{ for all }x,y\in\overline{\Omega},$$
i.e., $v$ is $L-Lipschitz$. Additionally, since  $v_{n,m}$ is $C_{m}$-semiconcave, we have, for all $x,y\in\Omega$ and $t\in[0,1]$,
$$ t v_{n,m}(x) + (1-t)v_{n,m}(y) - v_{n,m} (tx+(1-t)y) \le \frac{C_m}{2} t(1-t) |x-y|^2,$$
see \cite[Proposition 1.1.3]{Cannarsa2004}.
Applying the uniform convergence of $v_{n,m}$ to $v_n$ and \eqref{Theo:errorbounds:hyp2} we deduce, for all $x,y\in\Omega$ and $t\in[0,1]$, that
$$ t v_n(x) + (1-t)v_n(y) - v_n(tx+(1-t)y) \le \frac{C}{2} t(1-t) |x-y|^2.$$
which implies that $x\mapsto v_n(x)-\frac{C}{2}x^2$ is concave and consequently $v_n$ is $C-$semiconcave.

This allows us to apply \cref{prop:interpolation} in conjuntion with \Cref{lemma:BVnormSC} to $v_{n,m,\varepsilon}-v_{n}$ and obtain that \eqref{Theo:errorbounds:eq2:2:1} and \eqref{Theo:errorbounds:eq2:2:2} hold.

\end{proof}
\begin{rem}
The fact that this result  allows to measure the $L^p$-error between $\nabla v_n$ by $\nabla v_{n,m}$ in terms of the $C(\overline \Omega)$-error between $\phi_i$ and  $\phi_{i,m}$, is of importance for feedback control, since the feedback operator can be expressed in terms of the gradient of the value function. Semiconcavity of the value function is a well-studied property, see for instance \cite{Bardi1997,Cannarsa2004}.    The established estimate can also be relevant for obtaining error bounds which do not increase exponentially with the dimension. 
\end{rem}
We are now in position to prove the main theorem regarding the approximation properties of the proposed architecture.

\begin{theo}
\label{theo:ConvFiniteApprox}
    Let $v$ be a C-semiconcave function, which is Lipschitz continuous in $\overline{\Omega}$ with constant $L>0$, and suppose that Hypothesis \ref{hypo:approx} holds. Then for each $p\in[1,\infty)$ and each $\delta>0$,  there exist $\varepsilon>0$, $m\in\N$, $n\in\N$, and $\theta_m=(\theta_{m,1},\ldots,\theta_{n,m})\in\Theta_m^n$ such that
    we have
    $$ \norm{\varv_{n,m,\varepsilon}(\theta_{m})-v}_{C(\overline{\Omega})}+\norm{\nabla \varv_{n,m,\varepsilon}(\theta_{m})-\nabla v}_{L^{p}(\Omega;\R^d)}\leq \delta, $$
    $v_{n,m,\varepsilon}$ is ($L+\delta$) Lipschitz and $(C+\delta)$-semiconcave.
\end{theo}
\begin{proof}
    Let $p\in [1,\infty)$, $\delta>0$, and $ \{\phi\}_{i=1}^{\infty}$ be the family of function established in \Cref{prop:discreteRep}. For each $n\in\N$, let us denote
    $$ v_{n}=\min_{i=1,\ldots,n}\phi_i(x).$$
    By construction each $v_n$ is semiconcave with constant $C$ and Lipschitz continuous with constant $L$.
    By \Cref{prop:discreteRep} we can find $n\in\N$ such that
    \beq \norm{v_{n}-v}_{C(\overline{\Omega})}+\norm{\nabla v_{n}-\nabla v}_{L^{p}(\Omega;\R^d)}\leq \frac{\delta}{2}.
    \label{theo:ConvFiniteApprox:proof:eq1}
    \eeq
  Further  $\norm{v_{n}}_{W^{1,\infty}(\Omega)}\leq L$, and $v_n$ is $C-$semiconcave.
    
In addition, by \Cref{hypo:approx}, applied for $\xi= \phi_i$ with $i=\{1,\ldots,n\}$, and Theorem \ref{Theo:errorbounds} we have that there exists $\theta_m=(\theta_{m,1},\ldots,\theta_{n,m})\in\Theta_m^n$ and $\varepsilon>0$ satisfying 
    \beq    
    \norm{\varv_{n,m,\varepsilon}(\theta_m)-v_n}_{C(\overline{\Omega})}+\norm{\nabla \varv_{n,m,\varepsilon}(\theta)- \nabla v_n}_{L^p(\Omega;\R^d)}\leq \frac{\delta}{2},
    \label{theo:ConvFiniteApprox:proof:eq2}
    \eeq
$\norm{v_{n,m,\varepsilon}}_{W^{1,\infty}(\Omega)}\leq L+\delta$, and $v_{n,m,\varepsilon}$ is $(C+\delta)-$semiconcave.
Combining \eqref{theo:ConvFiniteApprox:proof:eq1}, and \eqref{theo:ConvFiniteApprox:proof:eq2}+ we arrive at

    $$  \norm{\varv_{n,m,\varepsilon}(\theta_m)-v}_{C(\overline{\Omega})}+\norm{\nabla \varv_{n,m,\varepsilon}(\theta)- \nabla v}_{L^p(\Omega;\R^d)}\leq  \delta,$$
    $\norm{\nabla \varv_{n,m,\varepsilon}(\theta_m)}_{C^1(\overline{\Omega})}\leq L+\delta$, and $\varv_{n,m,\varepsilon}(\theta_m)$ is $(C+\delta)$-semiconcave.
\end{proof}

We return to \Cref{Theo:errorbounds} and observe that it only connects the errors in the  norms of $C(\overline{\Omega})$ and $W^{1,p}(\Omega)$, with $p\in [1,\infty)$. However, for some applications it {{can be}} important to control the $W^{1,\infty}(\Omega)$ error. Since the function $v_n$ is only Lipschitz continuous it is not possible to obtain convergence in $W^{1,\infty}(\Omega)$ of $v_{n,m}$, since such a convergence would imply that $v_n$ is $C^1(\Omega)$.  Nevertheless, we can still identify regions where uniform convergence of the gradients  holds.  This involves  the set of active indices. For a given $x\in\Omega$, these are the indices for which $\min_{i=1,\ldots,n}\phi_i(x)$, respectively $\min_{i=1,\ldots,n}\phi_{i,m}(x)$, are achieved.

\begin{theo}
    \label{Theo:errorboundsLoc} 
    Let $\{\phi_{i,m}\}_{m}  \subset C^1(\overline{\Omega})$ be  sequences of functions converging to $\phi_i$
    in $C^1(\overline{\Omega})$ for $i=1,\ldots,n$. Further set
    $ v_n=\min_{i=1,\ldots,n}\phi_i,\ v_{n,m}=\min_{i=1,\ldots,n}\phi_{i,m},$
    and introduce  the sets of active indices for each $x\in \overline{\Omega}$
    $$I_n(x)=\{i:\phi_{i}(x)=v_n(x)\}, \ I_{n,m}(x)=\{i:\phi_{i,m}(x)=v_{n,m}(x)\},$$
    and for $\delta>0$ the sets
    $$ \Omega_{\delta}=\{x\in \overline{\Omega}: v_n(x)\leq \phi_j(x)-\delta\mbox{ for all }j\notin I_n(x)\}. $$
    Then, for all $m\in \N$ such that $\norm{v_{n}-v_{n,m}}_{C(\overline{\Omega})}< \frac{\delta}{2}$ we have that $I_{n,m}(x)\subset I_n(x)$, for $x\in \Omega_\delta$, and
    \beq  \norm{\nabla v_n -\nabla v_{n,m}}_{L^\infty(\Omega_{\delta};\R^{d})}\leq \sup_{i=1,\ldots,n}\norm{\nabla \phi_i-\nabla \phi_{i,m}}_{C(\overline{\Omega};\R^d)}.
\label{Theo:errorboundsLoc:cons1}    
\eeq
Additionally, defining $v_{n,\varepsilon}=\psi_{n,\varepsilon}(\phi_{1},\ldots,\phi_{n})$ and $v_{n,m,\varepsilon}=\psi_{n,\varepsilon}(\phi_{1,m},\ldots,\phi_{n,m})$, if \eqref{lemma:prob:hyp0} holds, then we have for $\varepsilon\in \left(0,\frac{\delta}{2(n-1)}\right)$ that
\beq 
\norm{\nabla v_n -\nabla v_{n,\varepsilon}}_{L^\infty(\Omega_{\delta};\R^{d})}\leq 2L(g_{\varepsilon}'(0)+(1-(g_\varepsilon'(\delta)^{(n-1)})))
\label{Theo:errorboundsLoc:cons2}
\eeq
and
\beq
\begin{array}{l}
\displaystyle\norm{\nabla v_n -\nabla v_{n.m,\varepsilon}}_{L^\infty(\Omega_{\delta};\R^{d})}\leq 2L(g_{\varepsilon}'(0)+(1-(g_\varepsilon'(\delta)^{(n-1)})))+\\
\ecart\displaystyle 2n(n-1)L\norm{g''_{\varepsilon}}_{L^\infty(\R)}\sup_{i=1,\ldots,n}\norm{\phi_i-\phi_{i,m}}_{C(\overline{\Omega})}+\sup_{i=1,\ldots,n}\norm{\nabla \phi_i-\nabla \phi_{i,m}}_{C(\overline{\Omega};\R^d)}.
\end{array}
\label{Theo:errorboundsLoc:cons3}
\eeq
where 
$$L=\sup_{i\in\{1,\ldots,n\}}\norm{\nabla \phi_i}_{C(\overline{\Omega};\R^d)}. $$
\end{theo}
\begin{proof}We first verify that $I_{n,m}(x)\subset I(x)$ for $x\in \Omega_{\delta}$.
    For $m\in\N$ let us set $$\delta_m:=\sup_{i=1,\ldots,n}\norm{\phi_{i,m}-\phi_i}_{C(\overline{\Omega})}.$$
Let $m\in\N$ be such that $\delta_m\leq \frac{\delta}{2}$ and consider $x\in \Omega_\delta$ and $i\in I_{n,m}(x)$. Then we have that
    $$ v_{n,m}(x)=\phi_{i,m}(x)\geq \phi_{i}(x)-\delta_m.$$
    Proceeding by contradiction, if  $i\not\in I_n(x)$, then
    $ v_n(x)\leq \phi_i(x)-\delta.$
    Combining these two inequalities we obtain
     $$ v_n(x)\leq v_{n,m}(x)-\delta+\delta_m.$$
    This implies that
    $v_{n,m}(x)-v_n(x)\geq\delta-\delta_m$ and hence
     $$|v_{n,m}(x)-v_n(x)|\geq\delta-\delta_m\ge \frac{\delta}{2}.$$

    On the other hand, we also know that $|v_{n}(x)-v_{n,m}(x)|< \frac{\delta}{2}$, which leads to
    a contradiction. Thus $I_{n,m}(x) \subset I_n(x)$ for $x\in \Omega_\delta$ follows.

    Turning to \eqref{Theo:errorboundsLoc:cons1}, let us first observe that $\Omega_\delta$ is cloased and there measurable. Further, both functions $v_n$ and $v_{n,m}$ are a.e. differentiable in $\Omega_\delta$.  Let $x$ be an arbitrary element of $\Omega_{\delta}$ where both functions are differentiable.  Choose $i\in I_{n,m}(x)$. Then by the previous step we have $i\in I_n(x)\cap I_{n,m}(x)$. This  implies that $v_n(x)=\phi_{i}(x)$ and $v_{n,m}(x)=\phi_{i,m}(x)$. By Theorems 3.2.13 and 3.2.2 in \cite[Part I, Section 3.1, pg.37,47]{Makela}, we have $\nabla v_n(x)=\nabla \phi_{i}(x)$ and $\nabla v_{n,m}(x)=\nabla\phi_{i,m}(x)$. Furthermore,  for $j\in I_{n,m}(x)$ 
\beq
\nabla v_{n,m}(x)=\nabla \phi_{j,m}(x)
\label{Theo:errorboundsLoc:eq1}
\eeq
    holds. Then it is clear that $$|\nabla v_n(x)-\nabla v_{n,m}(x)|\leq \sup_{j=1,\ldots,n}\norm{\nabla \phi_j-\nabla \phi_{j,m}}_{L^\infty(\Omega;\R^d)}.$$
    Since $v_{m}$ and $v$ are differentiable a.e. the above equality holds a.e. in $\Omega_\delta$ and  consequently
    $$ \norm{\nabla v_n-\nabla v_{n,m}}_{L^\infty(\Omega_\delta;\R^d)}\leq \sup_{i=1,\ldots,n}\norm{\nabla \phi_i-\nabla \phi_{i,m}}_{L^\infty(\Omega;\R^d)},$$
and thus \eqref{Theo:errorboundsLoc:cons1} is satisfied.

In order to prove \eqref{Theo:errorboundsLoc:cons2} we will first prove the following: 
\beq |1-p_{n,\hat{i},\varepsilon}(\Phi_{n}(x))|\leq g'_{\varepsilon}(0)+\left(1-g'_{\varepsilon}\left(\frac{\delta}{2}\right)^{n-1}\right),
\label{Theo:errorboundsLoc:eq6}
\eeq
for $x\in\Omega_{\delta}$ and $\hat{i}=\hat{i}_n(\Phi_n(x))$. For this purpose, from \eqref{eq:recursive:1} we deduce
\beq
\begin{array}{l}
\displaystyle |1-p_{n,\hat{i},\varepsilon}(\Phi_{n}(x))|=1-\prod_{j=\hat{i}+1}^ng'_{\varepsilon}(\phi_j(x)-v_{j-1,\varepsilon}(x))\left(1-g'_\varepsilon(\phi_{\hat i}(x)-v_{\hat{i}-1,\varepsilon}(x))\right) \\
\ecart\displaystyle 
=1-\prod_{j=\hat{i}+1}^ng'_{\varepsilon}(\phi_j(x)-v_{j-1,\varepsilon}(x))+g'_\varepsilon(\phi_{\hat{i}}(x)-v_{\hat{i}-1,\varepsilon}(x))\prod_{j=\hat{i}+1}^ng'_{\varepsilon}(\phi_j(x)-v_{j-1,\varepsilon}(x)).
\end{array}
\label{Theo:errorboundsLoc:eq7}
 \eeq
We analyze separately each of the terms in rightmost expression in the previous equality. Since $\hat{i}\in I_{n}(\Phi_n(x))$ we know that 
$$ \phi_{\hat{i}}(x)\leq v_{\hat{i}-1}(x).$$
Additionally, since \eqref{lemma:prob:hyp0} holds, we can use \Cref{lemma:upperboundPsi} in the previous inequality to deduce that 
$$\phi_{\hat{i}}(x)\leq v_{\hat{i}-1,\varepsilon}(x).$$
Since $g'_{\varepsilon}$ is monotonically increasing, we have 
\beq\label{eq:kk19}
g'_\varepsilon(\phi_{\hat{i}}(x)-v_{\hat{i}-1,\varepsilon}(x))\leq g'_{\varepsilon}(0).
\eeq
From this, and using \eqref{hypo:positivepart:2} we can bound the last term in \eqref{Theo:errorboundsLoc:eq7} and obtain
\beq 
g'_\varepsilon(\phi_{\hat{i}}(x)-v_{\hat{i}-1,\varepsilon}(x))\prod_{j=\hat{i}+1}^ng'_{\varepsilon}(\phi_j(x)-v_{j-1,\varepsilon}(x))\leq g'_{\varepsilon}(0). \label{Theo:errorboundsLoc:eq9}
\eeq
For the remaining term in \eqref{Theo:errorboundsLoc:eq7}, we note that since $\hat{i}=\max\{i\in I_{n}(\Phi_n(x))\}$ and $x\in\Omega_\delta$, we have that if $j\in\{\hat{i}+1,\ldots,n\}$, then 
$$\phi_j(x)-v_{j-1}(x)\geq \delta.$$
Using \eqref{prop:propertiesApprox:cons6} in the previous inequality and the fact that $\varepsilon\in \left(0,\frac{\delta}{2(n-1)}\right)$ we infer that 
$$\phi_j(x)-v_{j-1,\varepsilon}(x)>\frac{\delta}{2}.$$
Combining this with the  monotonicity of $g_{\varepsilon}'$ we deduce that 
$$g'_{\varepsilon}(\phi_j(x)-v_{j-1,\varepsilon}(x))\geq g'_{\varepsilon}\left(\frac{\delta}{2}\right).$$
This implies that 
\beq 
1-\prod_{j=\hat{i}+1}^ng'_{\varepsilon}(\phi_j(x)-v_{j-1,\varepsilon}(x))\leq 1-g'_{\varepsilon}\left(\frac{\delta}{2}\right)^{{{n-\hat{i}}}}\leq 1-g'_{\varepsilon}\left(\frac{\delta}{2}\right)^{n-1}.
\label{Theo:errorboundsLoc:eq8}
\eeq
Using \eqref{Theo:errorboundsLoc:eq9} and \eqref{Theo:errorboundsLoc:eq8} in  \eqref{Theo:errorboundsLoc:eq7}, we conclude that  \eqref{Theo:errorboundsLoc:eq6} holds. 

Assuming that $v_{n}$ is differentiable at $x$, we have  $\nabla v_n(x)=\nabla \phi_{\hat{i}}(x)$. Combining this with \eqref{eqkk4} and \eqref{Theo:errorboundsLoc:eq6} we obtain  

\begin{equation*}
\begin{array}{l}
\displaystyle
|\nabla v_{n,\varepsilon}(x)-\nabla v_n(x)| =\big |\nabla \phi_{\hat{i}}(x)(1-p_{n,\hat{i},\varepsilon}(\Phi_n(x)))+\sum_{\hat{i}\neq j}\nabla \phi_j(x)p_{n,j,\varepsilon}(\Phi_{n}(x))\big|
\\ \ecart
\displaystyle
\leq L(1-p_{n,\hat{i},\varepsilon}(\Phi_n(x)))+L\sum_{j\neq \hat{i}}p_{n,j,\varepsilon}(\Phi_{n}(x)) \leq 2L(1-p_{n,\hat{i},\varepsilon}(\Phi_n(x)))
\\ \ecart
\displaystyle
\leq 2L\left(g'_{\varepsilon}(0)+\left(1-g'_{\varepsilon}\left(\frac{\delta}{2}\right)^{n-1}\right)\right),
\end{array}
\end{equation*}
where in the next to last inequality we used \eqref{prop:propertiesApprox:cons3}. Using the almost everywhere differentiability of $v_{n}$ we obtain \eqref{Theo:errorboundsLoc:cons2}.

To prove \eqref{Theo:errorboundsLoc:cons3} we will first derive an upper bound of 
$$ \norm{\nabla v_{n,m,\varepsilon}-\nabla v_{n,m}}_{L^\infty(\Omega_\delta)}.$$
For this purpose, we note that by \eqref{prop:propertiesApprox:cons4} we have for all $x\in\overline{\Omega}$ 
\beq
\begin{array}{l}
\displaystyle (\nabla v_{n,m,\varepsilon}(x)-\nabla v_{n,\varepsilon}(x))^\top= D\psi_{n,\varepsilon}(\Phi_{n,m}(x)) D\Phi_{n,m}(x)-D\psi_{n,\varepsilon}(\Phi_{n}(x)) D\Phi_{n}(x) \\
\ecart\displaystyle
=(D\psi_{n,\varepsilon}(\Phi_{n,m}(x))-D\psi_{n,\varepsilon}(\Phi_{n}(x))) D\Phi_{n}(x)+ D\psi_{n,\varepsilon}(\Phi_{n,m}(x)) (D\Phi_{n,m}(x)-D\Phi_{n}(x)).
\label{Theo:errorboundsLoc:eq2}
\end{array}
\eeq
where $D$ denotes the Jacobian matrix. To bound the first term on the right hand side of \eqref{Theo:errorboundsLoc:eq2} we have 
\beq \begin{array}{l}
\displaystyle |D\psi_{n,\varepsilon}(\Phi_{n,m}(x))-D\psi_{n,\varepsilon}(\Phi_{n}(x))) D\Phi_{n}(x)|\leq |\nabla \psi_{n,\varepsilon}(\Phi_{n,m}(x))-\nabla \psi_{n,\varepsilon}(\Phi_{n}(x))|\cdot | D\Phi_{n}^\top (x)| \\
\ecart \displaystyle
\leq n^\frac{1}{2}L \norm{\nabla^2\psi_{n,\varepsilon} }_{L^\infty(\R^n;\R^{n\times n})}|\Phi_{n,m}(x)-\Phi_{n}(x)| \\
\ecart\displaystyle \leq nL \norm{\nabla^2\psi_{n,\varepsilon} }_{L^\infty(\R^n;\R^{n\times n})}\max_{i\in\{1,\ldots,n\}}\norm{ \phi_i- \phi_{i,m}}_{C(\overline{\Omega})},
\end{array} 
\label{Theo:errorboundsLoc:eq3}
\eeq
as for the second one we have 
\beq 
\begin{array}{l}
\displaystyle 
|D\psi_{n,\varepsilon}(\Phi_{n,m}(x)) (D\Phi_{n,m}(x)\!-\!D\Phi_{n}(x))|\!=\!\left(\sum_{i=1}^d\!\left(\sum_{j=1}^n\frac{\partial\psi_{n,\varepsilon}}{\partial  a_j}(\Phi_{n,m}(x))\!\left(\frac{\partial \phi_{j,m}}{\partial x_i}(x)\!-\!\frac{\partial \phi_{j}}{\partial x_i}(x)\!\right)\right)^2\right)^\frac{1}{2}\\
\ecart\displaystyle 
\leq \!\left(\sum_{i=1}^d\sum_{j=1}^n\frac{\partial\psi_{n,\varepsilon}}{\partial  a_j}(\Phi_{n,m}(x))\!\left(\frac{\partial \phi_{j,m}}{\partial x_i}(x)\!-\!\frac{\partial \phi_{j}}{\partial x_i}(x)\right)^2\right)^\frac{1}{2}\!=\!\left(\sum_{j=1}^n\frac{\partial\psi_{n,\varepsilon}}{\partial  a_j}(\Phi_{n,m}(x))|\nabla \phi_{j,m}(x)\!-\!\phi_{j}(x)|\right)^\frac{1}{2}\\
\ecart\displaystyle \leq \max_{i\in\{1,\ldots,n\}}\norm{\nabla \phi_{i}-\nabla \phi_{i,m}}_{C(\overline{\Omega};\R^d)},
\end{array}
\label{Theo:errorboundsLoc:eq4}
\eeq
where we have used Jensen's inequality. Plugging \eqref{Theo:errorboundsLoc:eq3} and \eqref{Theo:errorboundsLoc:eq4} in the right hand side of \eqref{Theo:errorboundsLoc:eq2} we arrive at
\beq
\begin{array}{l}
\displaystyle|(\nabla v_{n,m,\varepsilon}(x)-\nabla v_{n,\varepsilon}(x))|
\leq nL\norm{\nabla^2 \psi_{n,\varepsilon}}_{L^\infty(\R^n;\R^{n\times n} )}\max_{i\in\{1,\ldots,n\}}\norm{ \phi_i- \phi_{i,m}}_{C(\overline{\Omega})} +
\\
\ecart \displaystyle 
\max_{i\in\{1,\ldots,n\}}\norm{\nabla \phi_i-\nabla \phi_{i,m}}_{C(\overline{\Omega};\R^d)}.
\end{array}
\eeq
Using  \eqref{prop:propertiesApprox:cons4} on the first term in the right hand side of the above inequality we get
\beq
\begin{array}{l}
\displaystyle|(\nabla v_{n,m,\varepsilon}(x)-\nabla v_{n,\varepsilon}(x))|
\leq  2n(n-1)L \norm{g''_{\varepsilon}}_{L^\infty(\R)}\max_{i\in\{1,\ldots,n\}}\norm{ \phi_i- \phi_{i,m}}_{C(\overline{\Omega})}+ \\
\ecart \displaystyle 
\max_{i\in\{1,\ldots,n\}}\norm{\nabla \phi_i-\nabla \phi_{i,m}}_{C(\overline{\Omega};\R^d)}.
\end{array}
\label{Theo:errorboundsLoc:eq5}
\eeq
Finally, using \eqref{Theo:errorboundsLoc:eq5} and \eqref{Theo:errorboundsLoc:cons2}, we obtain  \eqref{Theo:errorboundsLoc:cons3} by the triangle inequality.
\end{proof}

\begin{rem}

\label{rem:errorboundsLoc:GradConvMoreau}

Let us discuss  some consequences of the previous theorem.  For this purpose we  observe that by  \eqref{hypo:positivepart:5} and  \eqref{lemma:prob:hyp1}  the term $(g_{\varepsilon}'(0)+(1-(g_\varepsilon'(\delta)^{(n-1)})))$ tends to $0$ as $\varepsilon \to 0.$ Moreover let us choose $\varepsilon_m\to 0^+$ such that in addition to \eqref{hypo:positivepart:6} the following holds:
$$ 
\lim_{m\to\infty}\left(\norm{g_{\varepsilon_m}''}_{L^\infty(\R)} \max_{i\in\{1,\ldots,n\}}\norm{\phi_{i,m}-\phi_i}_{C(\overline{\Omega})}\right)=0.$$
Then as a consequence of  \Cref{Theo:errorboundsLoc} we obtain that $v_{n,m,\varepsilon_m}$ converges to $\nabla v_n$ in $W^{1,\infty}(\Omega_{\delta})$ for each $\delta >0$, as $m\to \infty$. 

As a second consequence,  if we choose $g_{\varepsilon}=g_{\varepsilon,M}$ defined in Remark \ref{rem:kk1}, then we have that
$$ g'_{\varepsilon,M}(0)=0,\ g'_{\varepsilon}\left(\delta\right)=1\text{ and }\norm{g''}_{L^\infty(\R)}=\frac{1}{\varepsilon}, $$
if $\varepsilon<\delta $.
From this we deduce that under this choice the right hand side of  \eqref{Theo:errorboundsLoc:cons3} is bounded by \beq 2n(n-1)\frac{L}{\varepsilon}\sup_{i\in\{1,\ldots,n\}}\norm{\phi_{i,m}-\phi_i}_{C(\overline{\Omega})}+\sup_{i\in\{1,\ldots,n\}}\norm{\nabla \phi_{i,m} -\nabla \phi_i}_{C(\overline{\Omega};\R^d)}.
\label{eq:errorbound:gM}
\eeq
Consequently, for $\varepsilon\in \left(0,\frac{1}{2\delta(n-1)}\right)$ fixed as specified in Theorem \ref{Theo:errorboundsLoc}, $v_{n,m,\varepsilon}$ converges to $v_n$ in $W^{1,\infty}(\Omega_\delta)$ as $m\to\infty$.
\end{rem}

\begin{rem}
\label{rem:errorboundsLoc:Uconv}
It should be noted that \eqref{Theo:errorboundsLoc:cons3} does not imply convergence in $C(\Omega_\delta;\R^d)$ of  $\nabla v_{n,m,\varepsilon}$. However, noticing that the function $G(x)=\nabla \phi_{\hat{i}_n(x)}(x)$  is almost everywhere in $\Omega$ equal to $\nabla v_n$ and that 
$G$ is continuous in $\Omega_\delta$, we can infer, by \eqref{Theo:errorboundsLoc:cons3}, that 

\beq
\begin{array}{l}
\displaystyle\norm{G -\nabla v_{n.m,\varepsilon}}_{C(\Omega_{\delta};\R^{d})}\leq 2L(g'_{\varepsilon}(0)+(1-(g_\varepsilon'(\delta)^{(n-1)})))\\
\ecart\displaystyle + 2n(n-1)L\norm{g''_{\varepsilon}}_{L^\infty(\R)}\sup_{i=1,\ldots,n}\norm{\phi_i-\phi_{i,m}}_{C(\overline{\Omega})}+\sup_{i=1,\ldots,n}\norm{\nabla \phi_i-\nabla \phi_{i,m}}_{C(\overline{\Omega};\R^d)}.
\end{array}
\eeq
The continuity of $G$ in $\Omega_\delta$ is a consequence of the fact that  the active set of incides  $I_n$ does not change locally, that is, for $x\in\Omega_\delta$ we have  $I_n(y)=I_n(x)$  for all  $y\in B\left(x,\frac{\delta}{2L}\right)\cap \Omega$. To prove this assertion, we note that if $|y-x|<\frac{\delta}{2L}$, then for all $i\in\{1,\ldots,n\}$ we have $
|\phi_i(x)-\phi_i(y)|<\frac{\delta}{2}.$
By symmetry, it is enough to prove that $I_{n}(y)\subset I_{n}(x)$ for all $y\in B(x,\frac{\delta}{2L})\cap\Omega$. If $I_n(y)\nsubseteq  I_n(x)$ for some $y\in B(x,\frac{\delta}{2L})\cap\Omega$, then there exists $i\in I_n(y)\setminus I_{n}(x)$ and thus we obtain:
$$v(y)=\phi_{i}(y)\geq \phi_{i}(x)-\frac{\delta}{2} \geq v(x)+\frac{\delta}{2}\geq v(y)+\frac{\delta}{2}, $$
which is a contradiction.
\end{rem}

\section{Example}
\label{sec:Example}
In this section we introduce an example of a semiconcave function $v_d$ for $d\in\N$, which can be explicitly represented by a family of $2d$ functions of class $C^2$, satisfies a Hamilton Jacobi Bellman equation,  and which is not $C^1$. Additionally, we present a family of setting $S_{m}=(\Theta_m,\xi_m)$ for $m\in\N$ which satisfies \eqref{hypo:approx}. Utilizing this family of settings and the approximation of the positive part $g_{\varepsilon}=g_{\varepsilon,M}$ for $\varepsilon>0$, which was introduced in Remark \ref{rem:kk1}, we approximate $v_d$ by means of the the parametrization introduced in \eqref{def:parametrization}. 

In subsection \ref{subsec:setting} we verify that the hypotheses of Theorem \ref{Theo:errorbounds} are met and we use it to prove the convergence of the approximation. Furthermore, by means \ref{Theo:errorbounds} we will also prove that the the approximation is converging in $W^{1,\infty}(\Omega_\delta)$ for all $\delta>0$ and that the approximation satisfies the same Hamilton Jacobi Bellman equation approximately. 

In order to highlight the properties of the proposed approximation, we compere it with the approximation resulting from using the Log-Sum-Exp function (see Remark \ref{rem:kk2}) which is commonly utilized as a smooth approximation of the minimum. In contrast to the proposed approximation, the Log-Sum-Exp approximation is not able to deal with the discontinuities of the gradient and it struggles to solve the Hamilton Jacobi Bellman equation. To support this numerically, in subsection \ref{subsec:numexp} we implement both approximations and measure theirs convergences. Remarkably, these experiments suggest that the proposed approximation solves the Hamilton Jacobi Bellman equation approximately in a uniform sense.
\subsection{Exponential Distance Function}
In order to illustrate the properties of the proposed parametrization we introduce the Exponential Distance Function:
\beq 
x\in \overline\Omega \mapsto v_d(x):=\min_{i\in\{1,\ldots,2d\}}\phi_{i}(x)
\eeq
with  $\Omega=(-1,1)^d$ and
\beq 
\phi_{i}(x)=\exp\left(-\frac{1}{2}|x-e_i|^2\right)
\text{ and }
\phi_{d+i}(x)=\exp\left(-\frac{1}{2}|x+e_{i}|^2\right)
\eeq
for $i\in\{1,\ldots,d\}$ where $e_i$ is the $i$-th canocial vector of $\R^d$.

Clearly, the function $v_d$ is semiconcave thanks to \eqref{theo:SemiConcaveRepresentation} and it is a viscosity solution of the following Hamilton Jacobi Bellman equation:
\beq H(\nabla \phi(x),\phi(x)) =0\text{ for all }x\in\R^d.
\label{example:HJB}
\eeq
for the Hamiltonian $H:\R^d\times\R^+\to\R$ given by 
$$H(p,a)=\left\{\begin{array}{ll}
|p|^2+2\log(|a|)a^2 & \mbox{ if }a\neq 0 \\
|p|^2 & \mbox{ if }a=0
\end{array}\right. \mbox{, for }(p,a)\in\R^d\times \R.$$
Below we provide the proof.
\begin{lemma}
\label{lemma:ExpDist:HJB}
For each $i\in\{1,\ldots,2d\}$ the function $\phi_i$ satisfies
 \beq H(\nabla \phi_i(x),\phi_i(x))=0\mbox{ for all }x\in\R^d, \label{lemma:ExpDist:HJB:cons1}\eeq
 and 
 $v_{d}$ is a viscosity solution of \eqref{example:HJB}. 
\end{lemma}
\begin{proof}
We start by noting that an easy calculation proves that for each $i\in\{1,\ldots,n\}$, the function $\phi_i$ satisfies \eqref{lemma:ExpDist:HJB:cons1}. Additionally, fixing $x\in\R^d$ with $v_d$ differentiable at $x$, we have that $v_{d}(x)=\phi_i(x)$ and $\nabla v_{d}(x)=\nabla \phi_i(x)$ for some $i\in\{1,\ldots,n\}$ and hence we have that $v_d$ satsifies \eqref{example:HJB} at $x$. Then, due to the a.e. differentiability of $v_d$, we have that $v_d$ satisfies \eqref{example:HJB} a.e. and thus by Proposition 5.3.1 in \cite{Cannarsa2004} and the semiconcavity of $v_d$, we have that $v_d$ is a viscosity solution of \eqref{example:HJB}.  
\end{proof}
\begin{figure}[h!]
\centering
         \includegraphics[width=0.55\textwidth]{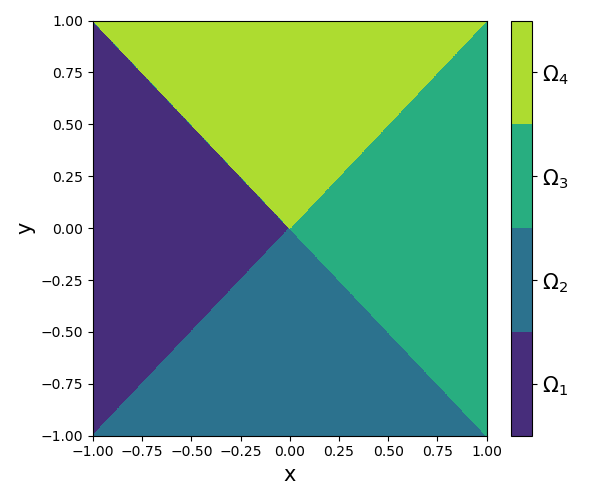}
         \caption{Active sets}
\label{ActiveSets}
\end{figure}

\begin{figure}[h!]
\centering
\begin{subfigure}[b]{0.45\textwidth}
\centering
         \includegraphics[width=\textwidth]{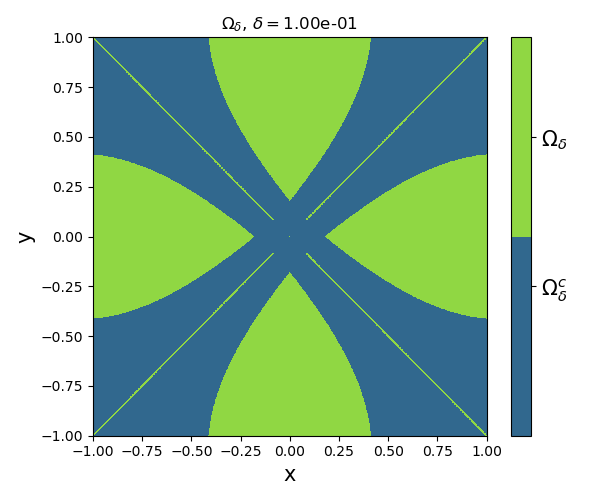}
         \caption{$\delta=10^{-1}$}
         \label{Omega4}
\end{subfigure}
\begin{subfigure}[b]{0.45\textwidth}
\centering
         \includegraphics[width=\textwidth]{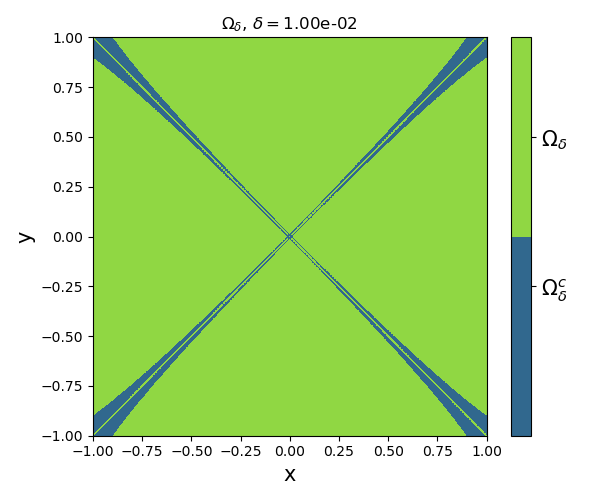}
         \caption{$\delta=10^{-2}$}
         \label{Omega3}
\end{subfigure}
\caption{$\Omega_{\delta}$ sets of $v_d$.}
\label{DeltaSets}
\end{figure}

In the remainder of this section, we use $d=2$ for the ease of the exposition, since it allows to depict the active sets and the set $\Omega_{\delta}$. In Figure \ref{ActiveSets} each active set is colored differently. From this we see that the discontinuities occurs along the diagonals of the square $[-1,1]^2$. Additionally, in Figure \ref{DeltaSets}, the set $\Omega_\delta$ is shown for $\delta\in\{10^{-2},10^{-1}\}$. As we see in \Cref{DeltaSets}, the sets $\Omega_{\delta}$ include part of the diagonals of the square $[-1,1]^2$, which is where the discontinuities of $\nabla v_{d}
$ occurs. To see this, we note that along the diagonals of $[-1,1]^2$ there is more than one active indices and $\nabla \phi_{i}\neq \nabla \phi_{j}$ for $i\neq j$. For example, if $x_{1}=x_{2}$ and $x_{1}\in[0,1]$, then $\phi_{3}(x_1,x_2)=\phi_{4}(x_1,x_2)$ and it can be verified that $\nabla v_{3}(x_{1},x_{2})\neq \nabla v_{4}(x_{1},x_{2})$.

\subsection{Setting and approximation}
\label{subsec:setting}

For $m\in\N$ we consider Chebyshev polynomials as setting, that is:
$$\Theta_m=\R^{(m+1)\times(m+1)} \mbox{ and }\xi_m(\theta)(x)=\sum_{i=1}^{m+1}\sum_{j=1}^{m+1}\theta_{i,j}T_{i-1}(x_1)T_{j-1}(x_2)\mbox{ for }x\in[-1,1]^2\mbox{ and }\theta\in\Theta_m,$$ 
where $T_{k}$ is the $k-$th Chebyshev polynomial for  $k\in\N_0$. By the results in \cite[Section 4.5.1 in Chapter 4]{Quarteroni} we see that the setting $S_{m}=(\Theta_m,\xi_m)$ satisfies \Cref{hypo:approx}. In particular, for $\phi\in C^\infty([-1,1]^2)$ we have that for coefficients $\theta_m\in\Theta_m$ obtained by interpolating $\phi$ at the Chebyshev points of degree $m$ (see \cite[Section 4.5.1 in Chapter 4]{Quarteroni})  we have 
$$ \lim_{m\to\infty}\xi_m(\theta_m)=\phi\mbox{ in }C^2([-1,1]^2).$$

Concerning the approximation of the positive part, we choose 
$$s\in\R\mapsto g_{\varepsilon,M}(s)=\left\{\begin{array}{ll}
0 &\text{if } s<0\\
\frac{s^2}{2\varepsilon} &\text{if } s\in [0,\varepsilon)\\
s-\frac{\varepsilon}{2} &\text{if } s>\varepsilon.
\end{array}\right.$$
According to Remarks \ref{rem:kk1} and \ref{rem:kk2}, $g_{\varepsilon,M}$ satisfies \eqref{hypo:positivepart}, \eqref{lemma:prob:hyp0} and \eqref{lemma:prob:hyp1}. We name the resulting approximation of the minimum $\psi_{n,\varepsilon}$ the \emph{MoreauRegMin} since it comes from the Moreau regularization of the positive part.

Defining $\theta_m=(\theta_{1,m},\ldots,\theta_{2d,m})$ with $\theta_{i,m}\in\Theta_{m}$ as the coefficients obtained by interpolating $\phi_i$ by Chebychev polynomials of total degree less or equal than $m$, for $i\in \{1,\ldots,2d\}$ we propose the following approximation for $v_{d}$:
$$v_{d,m,\varepsilon}=\varv_{2d,m,\varepsilon}(\theta_m),$$
We will refer to this approximation of $v_d$ as the \emph{MoreauRegMin approximation}.

By \eqref{Theo:errorbounds} we have
$$ \lim_{\varepsilon\to 0^+,m\to\infty}\norm{ v_{d,m,\varepsilon}-v_d}_{C(\overline{\Omega})}+\norm{\nabla v_{d,m,\varepsilon} -\nabla v_d}_{L^1(\Omega);\R^d}=0.$$
Furthermore, for $\delta>0$ and $\varepsilon\in (0,\frac{\delta}{2(2d-1)})$ fixed,  we have that 
$$\lim_{m\to\infty}\norm{v_{d,m,\varepsilon}-v_d}_{W^{1,\infty}(\Omega_\delta)}=0.$$
In particular, due the continuity of $H$ and the fact that $\varv_{2d,m,\varepsilon}(\theta_m)$ is of class $C^1$, we have that $$ 
\lim_{\varepsilon\to 0^+,m\to\infty}\norm{H(\nabla v_{d,m,\varepsilon},v_{d,m,\varepsilon})}_{C(\overline{\Omega_\delta})}=0.
$$
which in turns implies that 
$$\lim_{\varepsilon\to 0^+,m\to\infty}H(\nabla v_{d,m,\varepsilon}(x),v_{d,m,\varepsilon}(x))=0 $$ 
for all $x\in \overline{\Omega}$, since clearly 
$$\overline{\Omega}=\lim_{\tilde{\delta}\to0^+}\Omega_{\tilde{\delta}}=\bigcup_{\tilde{\delta}>0}\Omega_{\tilde{\delta}}.$$
It is noteworthy that the last three properties described above are not necessarily satisfied by a general approximation of $v_d$. For example, using the Log-Sum-Exp functions defined in \Cref{rem:kk2}, the approximation of $v_{d}$ given by 
\beq \tilde{v}_{d,m,\varepsilon}=\tilde{\psi}_{n,\varepsilon}(\xi_{\theta_{1,m}},\ldots,\xi_{\theta_{2d,m}}) 
\label{def:LogSumApprox}
\eeq
converges uniformly to $v_d$ as $\varepsilon\to 0^+$ and $m\to\infty$, it is Lipschitz and semiconcave uniformly with respect to $m\in\N$, but its gradient does not converges pointwise as is explained in \Cref{rem:kk2} and furthermore $H(\nabla \tilde{v}_{d,m,\varepsilon},\tilde{v}_{d,m,\varepsilon} )$ does not vanishes as $\varepsilon\to 0^+$ and $m\to\infty$ pointwise nor uniformly in $\Omega_\delta$ for any $\delta>0$. In particular, if we consider $x=\left(-\frac{1}{2},-\frac{1}{2}\right)$  we have that $v_{d}(x)=\phi_1(x)=\phi_2(x)=\exp\left(-\frac{5}{4}\right)$ and  
$$\nabla \phi_{1}(x)=\left(
\begin{array}{c}
3\\
1
\end{array}\right)
\frac{\exp\left(-\frac{5}{4}\right)}{2}\neq \left(
\begin{array}{c}
1 \\
3
\end{array}\right)
\frac{\exp\left(-\frac{5}{4}\right)}{2}=\nabla \phi_{2}(x),$$
choosing $\varepsilon_m=|\phi_{1,m}(x)-\phi_{2,m}(x)|^{\frac{1}{2}}$ we have 
 
$$\lim_{m\to\infty}\nabla \tilde{v}_{d,m,\varepsilon_m}(x)=\frac{1}{2}(\nabla \phi_1(x)+\nabla \phi_2(x))=\left(
\begin{array}{c}
1 \\
1
\end{array}\right)
\exp\left(-\frac{5}{4}\right),$$
which by the continuity of $H$ implies that 
$$\lim_{m\to\infty}H(\nabla \tilde{v}_{d,m,\varepsilon_m}(x),\tilde{v}_{d,m,\varepsilon_m}(x))=-\frac{1}{2}\exp\left(-5\right)<0. $$
In the next subsection, we compare numerically the convergence of $v_{d,m,\varepsilon}$ and $\tilde{v}_{d,m,\varepsilon}$ as $\varepsilon\to 0^+$ and $m\to\infty$.

\subsection{Numerical experiments}
\label{subsec:numexp}
To carry out the numerical experiments we consider different degrees of the Chebyshev interpolation and  regularization parameters $\varepsilon>0$:
$$m\in \{2,4,6,8,10\} \mbox{ and }\varepsilon\in\{10^{-4},10^{-2},10^{-1}\}.$$
We will gauge the convergence of the different approximations of $v_d$ in $\Omega_\delta$ for $\delta\in\{0,10^{-1},10^{-3}\}$, where  in the case $\delta=0$  we take that $\Omega_\delta=\overline{\Omega}$. For this, we consider a uniform grid $\mathcal{X}=\{(x_i,y_j)\}_{i,j=0}^{1000}$ of $[-1,1]^2$,  where $\{x_{i}\}_{i=0}^{1000}$ is a uniform division of $[-1,1]$. We measure the convergence using the following metrics:
$$ D_{C}(u,\delta)=\frac{1}{|\mathcal{X}\cap\Omega_\delta|}\max_{x\in\mathcal{X}\cap\Omega_\delta }|v_d(x)-u(x)|,$$
$$ D_{W1}(u,\delta)=\frac{1}{|\mathcal{X}\cap\Omega_\delta|}\sum_{x\in\mathcal{X}\cap\Omega_\delta}|\nabla v_d(x)-\nabla u(x)|,$$
$$ D_{W\infty}(u,\delta)=\max_{x\in\mathcal{X}\cap\Omega_\delta}|\nabla v_d(x)-\nabla u(x)|,$$
where $|\cdot|$ denotes the cardinality of a set as well as the $\R^n-$norm as appropriate.
Additionally, to verify the convergence of the Hamiltonian of the approxmiation we use:
$$ D_{H1}(u,\delta)=\frac{1}{|\mathcal{X}\cap\Omega_\delta|}\sum_{x\in\mathcal{X} \cap\Omega_\delta}|H(\nabla u(x),u(x))|, $$
$$ D_{H\infty}(u,\delta)=\max_{x\in \mathcal{X}\cap\Omega_\delta}|H(\nabla u(x_i,y_j),u(x))|.$$

\begin{table}[h!]
\centering
\begin{tabular}{|c|cccc||} 
 \hline
 $\delta$ & $10^{-4}$ & $10^{-3}$ & $10^{-2}$ & $10^{-1}$ \\ [0.5ex] 
 \hline
 
 $\frac{|\mathcal{X}\cap \Omega_\delta|}{|\mathcal{X}|}$ &1& 0.993& 0.918& 0.416 \\ [1ex] 
 \hline
\end{tabular}
\caption{Proportion of $\frac{|\mathcal{X}\cap \Omega_\delta|}{|\mathcal{X}|}$.}
\label{table:1}
\end{table}

\begin{figure}[h!]
\centering
\begin{subfigure}{0.45\textwidth}
\includegraphics[width=\textwidth]{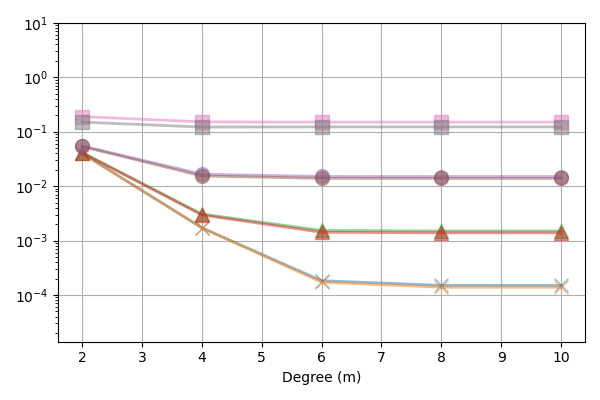}
\caption{$D_{C}(\cdot,\delta)$ error for $\delta=0$.}
\label{fig:DC0}
\end{subfigure}
\begin{subfigure}[b]{0.45\textwidth}
\includegraphics[width=\textwidth]{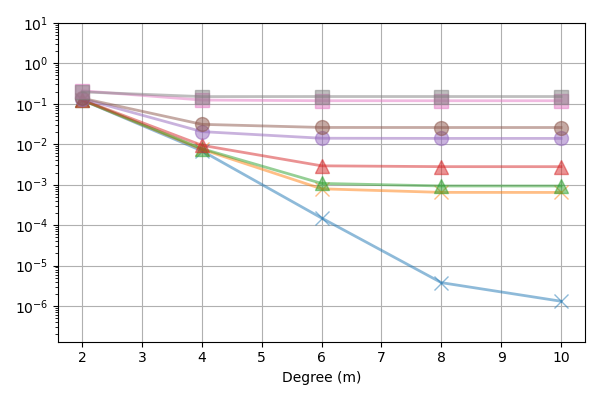}
\caption{$D_{W1}(\cdot,\delta)$ error for $\delta=0$.}
\label{fig:W10}
\end{subfigure}
\begin{subfigure}[b]{0.45\textwidth}
         \includegraphics[width=\textwidth]{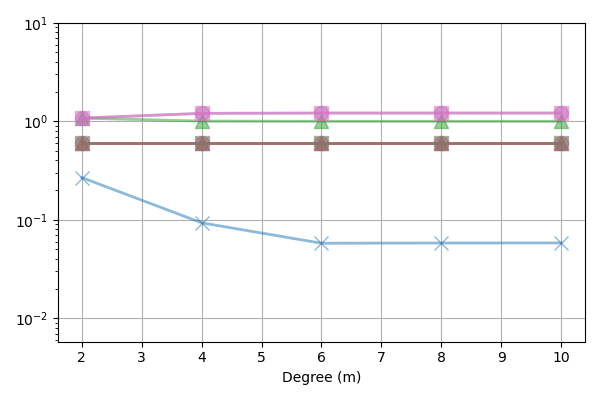}
\caption{$D_{W\infty}(\cdot,\delta)$ error for $\delta=0$.}
\label{fig:Dinfty0}
\end{subfigure}
\begin{subfigure}[b]{0.45\textwidth}
\includegraphics[width=\textwidth]{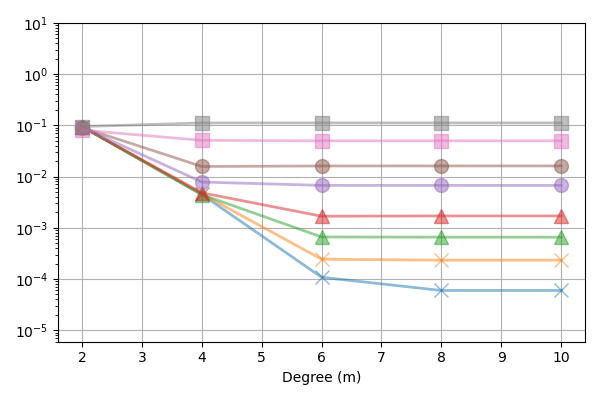}
\caption{$D_{H1}(\cdot,\delta)$ error for $\delta=0$.}
\label{fig:DH10}
\end{subfigure}
\begin{subfigure}[b]{0.45\textwidth}
\includegraphics[width=\textwidth]{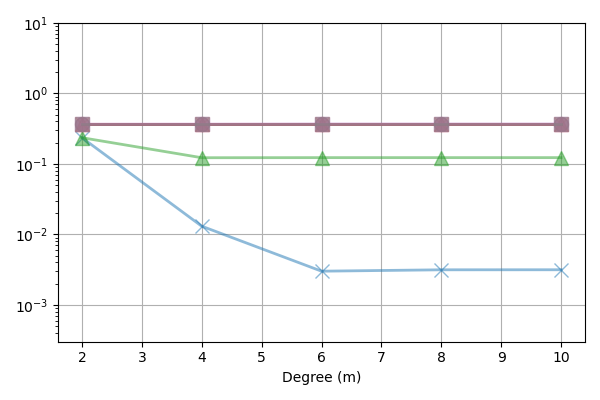}
\caption{$D_{H\infty}(\cdot,\delta)$ error for $\delta=0$.}
\label{fig:DHinf0}
\end{subfigure}
\begin{subfigure}[b]{0.45\textwidth}
         \includegraphics[width=\textwidth]{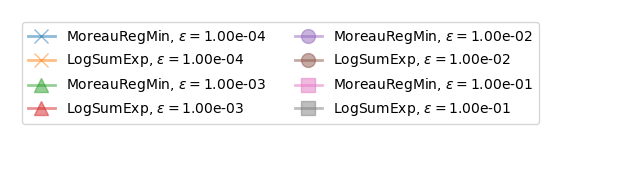}
\caption{Legend.}
\label{fig:Dinfty0}
\end{subfigure}
\caption{Error for $\delta=0$.}
\label{fig:error0}
\end{figure}

\begin{figure}
\centering
\begin{subfigure}[b]{0.45\textwidth}
         \includegraphics[width=\textwidth]{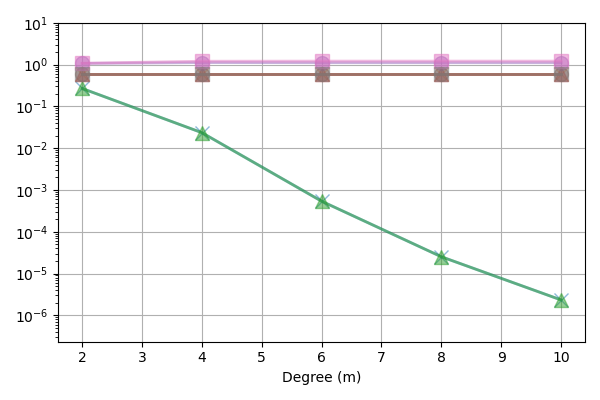}
\caption{$D_{W\infty}(\cdot,\delta)$ error for $\delta=10^{-2}$.}
\label{fig:Dinfty2}
\end{subfigure}
\begin{subfigure}[b]{0.45\textwidth}
\includegraphics[width=\textwidth]{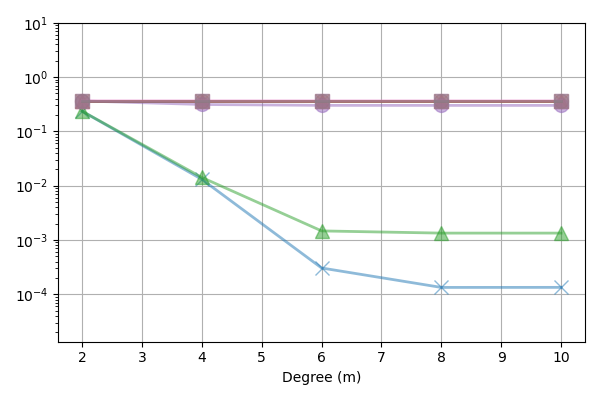}
\caption{$D_{H\infty}(\cdot,\delta)$ error for $\delta=10^{-2}$.}
\label{fig:DHinf2}
\end{subfigure}
\caption{Error for $\delta=10^{-2}$}
\label{fig:error2}
\end{figure}

\begin{figure}[h!]
\centering
\begin{subfigure}[b]{0.45\textwidth}
         \includegraphics[width=\textwidth]{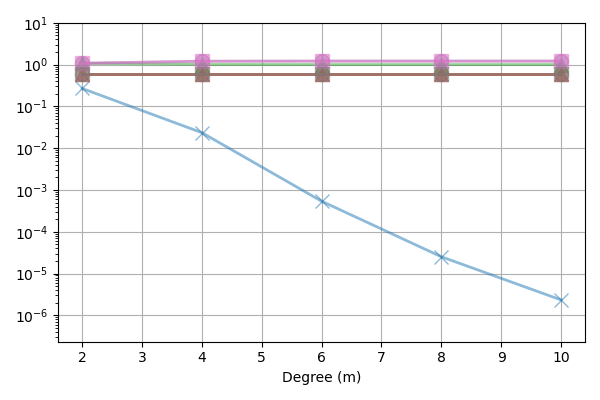}
\caption{$D_{W\infty}(\cdot,\delta)$ error for $\delta=10^{-3}$.}
\label{fig:Dinfty3}
\end{subfigure}
\begin{subfigure}[b]{0.45\textwidth}
\includegraphics[width=\textwidth]{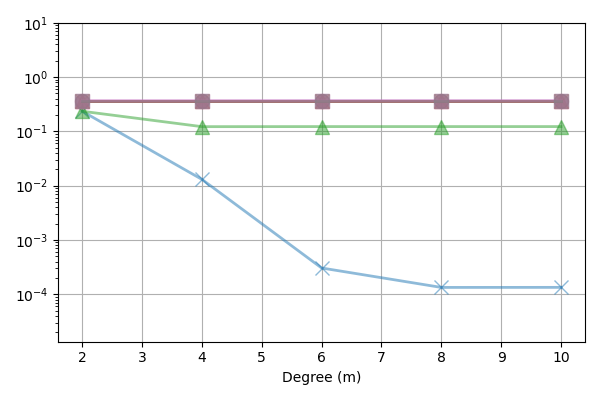}
\caption{$D_{H\infty}(\cdot,\delta)$ error for $\delta=10^{-3}$.}
\label{fig:DHinf3}
\end{subfigure}
\caption{Error for $\delta=10^{-3}$}
\label{fig:error3}
\end{figure}

The results are summarized in Table \ref{table:1}, where the proportion of the grid that intersects $\Omega_\delta$, for each considered $\delta$ is recorded, and in Figures \ref{fig:error0}-\ref{fig:error3},  where the behavior of the errors for $\delta\in\{0,10^{-3},10^{-2},10^{-1}\}$ are depicted. In the following, we discuss these results.

We observe in Figure \ref{DeltaSets} that  $\Omega_{\delta}$ covers part of the diagonals of the square. As $\delta\to 0^+$ the proportion of the diagonals covered by $\Omega_\delta$ increases and for $\delta=10^{-2}$, at least visually, it seems that this proportion  is close to 1. It is worth recalling that the discontinuities of the gradient of $v_d$ occur precisely along the diagonals of the square. In this example, even if $\delta$ is not extremely close to $0$, we can obtain a good representation of  these discontinuities.

Regarding the behavior of the errors in the case of $\delta=0$, and thus   $\Om_\delta=\overline{\Omega}$, from Figure \ref{fig:error0} we see that  the behavior of the $D_C$, $D_{W\infty}$, $D_{H\infty}$ is the same for the MoreuRegMin and the Log-Sum-Exp approximation. The main differences arise from the choice of $\varepsilon$.  In contrast, in the case of $D_{H1}$ and $D_{W1}$, independently of $\varepsilon$ and the degree, MoreuRegMin achieves a better performance in terms of $D_{W1}$. This is consistent with the fact that MoreuRegMin is able to represent in a better manner the gradient of $v_d$ and viscosity solutions of HJB equations.

For the cases $\delta=10^{-2}$ and $\delta=10^{-3}$, $D_{C}$, $D_{W1}$ and $D_{H1}$ do not differ from what was already shown in Figure \ref{fig:error0}. Thus we decided in Figures \ref{fig:error2} and \ref{fig:error3} to focus on the behavior of $D_{H\infty}$ and $D_{W\infty}$. It is noteworthy that for all $\varepsilon$ and degrees, both $D_{W\infty}$ and $D_{H\infty}$ are above $10^{-1}$ for the Log-Sum-Exp approximation, this indicates that this approximation is not able to approximate the gradient of $v_d$ uniformly. In contrast, we see in Figure \ref{fig:error2} that for $\delta=10^{-2}$, the error  $D_{W\infty}$ of the  MoreuRegMin approximation   decreases with the degree. The cases  $\varepsilon=10^{-3}$ and $\varepsilon=10^{-4}$ behave identically to each other.  As explained at the end of Section \ref{sec:approximation}, this is due to the fact that $g'_{\varepsilon,M}(0)=g'_{\varepsilon,M}(\frac{\delta}{2})=0$ if $\varepsilon<\delta$. In the case of $D_{H\infty}$, the behavior for  $\varepsilon=10^{-3}$ and $\varepsilon=10^{-4}$ is not identical because the Hamiltonian also receives as input the function itself for which the approximation error depends on $\varepsilon$. In agreement with these observations, in Figure \ref{fig:error3} for $\delta=10^{-3}$ we see that 
the only case in which the errors $D_{W\infty}$ and $D_{H\infty}$ decrease and achives values below $10^{-1}$ is when $\varepsilon=10^{-4}$ for the MoreuRegMin approximation. As in the case of $\delta=10^{-2}$ this is explained by the fact that $D_{W\infty}$ does not depend on $\varepsilon$ if $\varepsilon<\delta$.
\section{Conclusions}

\label{sec:conclusion}
In this work a smooth, semiconcavity preserving, approximation  was introduced. The proposed semiconcavity preserving approximation was devised with two components: a smooth approximation $\psi_{n,\varepsilon}$ of the minimum of $n-$ real-valued functions   and a universal approximating parametrization of $C^2$ functions. The universality of the  semiconcavity preserving approximation for semiconcave functions was proved  in the $C(\overline{\Omega})$ and $W^{1,p}(\Omega)$ norms for $p\in[1,\infty)$. Further uniform error bounds were presented for the gradient of the proposed approximation in a family of sets converging to $\Omega$,  which includes part of the dicontinuities of the gradient. Such a  result is not evident since the gradient of a semiconcave function is not necessarily continuous.

By analyzing the limiting behavior of the gradient of $\psi_{n,\varepsilon}$ as $\varepsilon\to 0^+$, the pertinence of this approximation for viscosity solutions of HJB equations was shown in Remark \ref{rem:prob:HJB}. It should be noted that although there exist other smooth approximations of the minimum function such as the Log-Sum-Exp, they do not necessarily exhibit the same behavior.

To illustrate the results of this article, in Section \ref{sec:Example} the exponential distance function was introduced. It is semiconcave and satisfies a HJB equation. By means of Chebyshev polynomials, a setting satisfying Hypothesis \ref{hypo:approx} was presented in Subsection \ref{subsec:setting}. The uniform convergence of the Hamiltonians in $\Omega_\delta$ towards 0 was proved  for this approximation, and for $\varepsilon$ small enough, the numerical results using the MoreauRegMin framework  confirm the behavior. 
 On the other hand, the Log-Sum-Exp approximation is not able to achieve the same results.

To conclude, we expect to utilize the proposed parametrization for devising new machine learning based methodologies for the resolution of HJB equations and the design of optimal feedback laws in future work. For the aforementioned task, it has been proved in \cite{KuVa2} that the semiconcavity of the approximation plays an important role for ensuring  its performance. In addition, we believe that the analysis of the parametrization developed in this article can help to understand under which conditions one can provide a parametrization which mitigates the so called curse of dimensionality.

\bibliography{biblio}

\begin{thebibliography}{29}
\providecommand{\natexlab}[1]{#1}
\providecommand{\url}[1]{\texttt{#1}}
\expandafter\ifx\csname urlstyle\endcsname\relax
  \providecommand{\doi}[1]{doi: #1}\else
  \providecommand{\doi}{doi: \begingroup \urlstyle{rm}\Url}\fi

\bibitem[Akian et~al.(2008)Akian, Gaubert, and Lakhoua]{Akian}
M.~Akian, S.~Gaubert, and A.~Lakhoua.
\newblock The max-plus finite element method for solving deterministic optimal
  control problems: Basic properties and convergence analysis.
\newblock \emph{SIAM J. Control Optim.}, 47\penalty0 (2):\penalty0 817--848,
  2008.

\bibitem[Albano(2009)]{Albano}
P.~Albano.
\newblock The singularities of the distance function near convex boundary
  points.
\newblock \emph{Nonlinear differ. equ. appl.}, 16:\penalty0 273--281, 2009.
\newblock \doi{10.1007/s00030-008-7035-y}.

\bibitem[Amos et~al.(2017)Amos, Xu, and Kolter]{AmosXuKolter}
B.~Amos, L.~Xu, and J.~Z. Kolter.
\newblock Input convex neural networks.
\newblock In D.~Precup and Y.~W. Teh, editors, \emph{Proceedings of the 34th
  International Conference on Machine Learning}, volume~70 of \emph{Proceedings
  of Machine Learning Research}, pages 146--155, New York, 06--11 Aug 2017.
  PMLR.
\newblock URL \url{https://proceedings.mlr.press/v70/amos17b.html}.

\bibitem[Bardi and Capuzzo-Dolcetta(1997)]{Bardi1997}
M.~Bardi and I.~Capuzzo-Dolcetta.
\newblock \emph{Optimal Control and Viscosity Solutions of
  Hamilton-Jacobi-Bellman Equations}.
\newblock Birkh\"{a}user, Boston, 1997.

\bibitem[Boyd and Magnani(2009)]{BoydMagnani}
S.~Boyd and A.~Magnani.
\newblock Convex piecewise-linear fitting.
\newblock \emph{Optim. Eng.}, 10:\penalty0 1--17, 2009.
\newblock \doi{10.1007/s11081-008-9045-3}.

\bibitem[Calafiore et~al.(2020)Calafiore, Gaubert, and
  Possieri]{CalafioreGaubertPossieri}
G.~C. Calafiore, S.~Gaubert, and C.~Possieri.
\newblock Log-sum-exp neural networks and posynomial models for convex and
  log-log-convex data.
\newblock \emph{IEEE Transactions on Neural Networks and Learning Systems},
  31\penalty0 (3):\penalty0 827--838, 2020.
\newblock \doi{10.1109/TNNLS.2019.2910417}.

\bibitem[Cannarsa and Sinestrari(2006)]{Cannarsa2004}
P.~Cannarsa and C.~Sinestrari.
\newblock \emph{Semiconcave functions, {Hamilton-Jacobi} equations, and optimal
  control}.
\newblock Progress in Nonlinear Differential Equations and Their Appli.
  Birkhauser, Boston, Feb. 2006.

\bibitem[Chen et~al.(2019)Chen, Shi, and Zhang]{ChenShiZhang}
Y.~Chen, Y.~Shi, and B.~Zhang.
\newblock Optimal control via neural networks: A convex approach.
\newblock In \emph{International Conference on Learning Representations}, 2019.
\newblock URL \url{https://openreview.net/forum?id=H1MW72AcK7}.

\bibitem[Darbon et~al.(2023)Darbon, Dower, and Meng]{DarbonDowerMeng}
J.~Darbon, P.~M. Dower, and T.~Meng.
\newblock Neural network architectures using min-plus algebra for solving
  certain high-dimensional optimal control problems and {Hamilton--Jacobi}
  {PDEs}.
\newblock \emph{Math. Control Signals Systems}, 35\penalty0 (1):\penalty0
  1--44, Mar. 2023.

\bibitem[Douglis(1965)]{Douglis}
A.~Douglis.
\newblock Solutions in the large for multi-dimensional non linear partial
  differential equations of first order.
\newblock \emph{Ann. I. Fourier}, 15\penalty0 (2):\penalty0 1--35, 1965.
\newblock URL \url{http://eudml.org/doc/73873}.

\bibitem[Dower et~al.(2015)Dower, McEneaney, and Zhang]{Dower}
P.~M. Dower, W.~M. McEneaney, and H.~Zhang.
\newblock Max-plus fundamental solution semigroups for optimal control
  problems.
\newblock In \emph{2015 Proceedings of the Conference on Control and its
  Applications}, pages 368–--375, 2015.

\bibitem[Evans and Gariepy(2015)]{Evans2}
L.~C. Evans and R.~Gariepy.
\newblock \emph{Measure Theory and Fine Properties of Functions}.
\newblock Chapman and Hall/CRC, New York, revised edition (1st ed.) edition,
  2015.
\newblock \doi{10.1201/b18333}.

\bibitem[Fleming(1969)]{Flemming}
W.~H. Fleming.
\newblock The cauchy problem for a nonlinear first order partial differential
  equation.
\newblock \emph{J. Differ. Equ.}, 5\penalty0 (3):\penalty0 515--530, 1969.
\newblock ISSN 0022-0396.
\newblock \doi{https://doi.org/10.1016/0022-0396(69)90091-6}.
\newblock URL
  \url{https://www.sciencedirect.com/science/article/pii/0022039669900916}.

\bibitem[Gao and Pavel(2018)]{GaoPavel}
B.~Gao and L.~Pavel.
\newblock On the properties of the softmax function with application in game
  theory and reinforcement learning, 2018.
\newblock URL \url{https://arxiv.org/abs/1704.00805}.

\bibitem[Gaubert et~al.(2011)Gaubert, McEneaney, and Qu]{Gaubert}
S.~Gaubert, W.~McEneaney, and Z.~Qu.
\newblock Curse of dimensionality reduction in max-plus based approximation
  methods: Theoretical estimates and improved pruning algorithms.
\newblock In \emph{2011 50th IEEE Conference on Decision and Control and
  European Control Conference}, pages 1054--1061, 2011.

\bibitem[G{\'e}n{\'e}rau et~al.(2022)G{\'e}n{\'e}rau, Oudet, and
  Velichkov]{Gener}
F.~G{\'e}n{\'e}rau, {\'E}.~Oudet, and B.~Velichkov.
\newblock Cut locus on compact manifolds and uniform semiconcavity estimates
  for a variational inequality.
\newblock \emph{Archive for Rational Mechanics and Analysis}, 246\penalty0
  (2--3):\penalty0 561--602, 2022.
\newblock \doi{10.1007/s00205-022-01821-0}.

\bibitem[Ghosh et~al.(2020)Ghosh, Pananjady, Guntuboyina, and
  Ramchandran]{GhoshPananjady}
A.~Ghosh, A.~Pananjady, A.~Guntuboyina, and K.~Ramchandran.
\newblock Max-affine regression with universal parameter estimation for
  small-ball designs.
\newblock In \emph{2020 IEEE International Symposium on Information Theory
  (ISIT)}, pages 2706--2710, 2020.
\newblock \doi{10.1109/ISIT44484.2020.9174116}.

\bibitem[Goujon et~al.(2024)Goujon, Neumayer, and Unser]{Goujon}
A.~Goujon, S.~Neumayer, and M.~Unser.
\newblock Learning weakly convex regularizers for convergent
  image-reconstruction algorithms.
\newblock \emph{SIAM J. Imaging Sci.}, 17\penalty0 (1):\penalty0 91--115, 2024.
\newblock \doi{10.1137/23M1565243}.
\newblock URL \url{https://doi.org/10.1137/23M1565243}.

\bibitem[Hanin(2019)]{Hanin}
B.~Hanin.
\newblock Universal function approximation by deep neural nets with bounded
  width and relu activations.
\newblock \emph{Mathematics}, 7\penalty0 (10), 2019.
\newblock ISSN 2227-7390.
\newblock \doi{10.3390/math7100992}.
\newblock URL \url{https://www.mdpi.com/2227-7390/7/10/992}.

\bibitem[Kružkov(1975)]{Kruzkov}
S.~N. Kružkov.
\newblock Generalized solutions of the hamilton-jacobi equations of eikonal
  type. i. formulation of the problems; existence, uniqueness and stability
  theorems; some properties of the solutions.
\newblock \emph{Math. USSR-Sb.}, 27\penalty0 (3):\penalty0 406, apr 1975.
\newblock \doi{10.1070/SM1975v027n03ABEH002522}.
\newblock URL \url{https://dx.doi.org/10.1070/SM1975v027n03ABEH002522}.

\bibitem[Kunisch and V\'asquez-Varas(2025)]{KuVa3}
K.~Kunisch and D.~V\'asquez-Varas.
\newblock Convergence of machine learning methods for feedback control laws:
  averaged feedback learning scheme and data driven methods, 2025.
\newblock URL \url{https://arxiv.org/abs/2407.18403}.

\bibitem[Kunisch and Vásquez-Varas(2024)]{KuVa2}
K.~Kunisch and D.~Vásquez-Varas.
\newblock Consistent smooth approximation of feedback laws for infinite horizon
  control problems with non-smooth value functions.
\newblock \emph{J. Differ. Equ.}, 411:\penalty0 438--477, 2024.
\newblock ISSN 0022--0396.
\newblock \doi{https://doi.org/10.1016/j.jde.2024.08.010}.
\newblock URL
  \url{https://www.sciencedirect.com/science/article/pii/S0022039624004923}.

\bibitem[Mäkelä and Neittaanmäki(1992)]{Makela}
M.~Mäkelä and P.~Neittaanmäki.
\newblock \emph{Nonsmooth Optimization}.
\newblock World Scientific, New Jersey, 1992.
\newblock \doi{10.1142/1493}.
\newblock URL \url{https://www.worldscientific.com/doi/abs/10.1142/1493}.

\bibitem[Quarteroni and Valli(2008)]{Quarteroni}
A.~Quarteroni and A.~Valli.
\newblock \emph{Numerical approximation of partial differential equations}.
\newblock Springer series in computational mathematics. Springer, Berlin,
  Germany, 1 edition, Sept. 2008.

\bibitem[Rifford(2000)]{Rifford}
L.~Rifford.
\newblock Existence of {L}ipschitz and semiconcave control-{L}yapunov
  functions.
\newblock \emph{SIAM J. Control Optim.}, 39\penalty0 (4):\penalty0 1043--1064,
  2000.
\newblock \doi{10.1137/S0363012999356039}.

\bibitem[Rifford(2002)]{Rifford2}
L.~Rifford.
\newblock Semiconcave control-{L}yapunov functions and stabilizing feedbacks.
\newblock \emph{SIAM J. Control Optim.}, 41\penalty0 (3):\penalty0 659--681,
  2002.
\newblock \doi{10.1137/S0363012900375342}.

\bibitem[Santambrogio(2015)]{santam}
F.~Santambrogio.
\newblock \emph{Optimal Transport for Applied Mathematicians: Calculus of
  Variations, PDEs, and Modeling}.
\newblock Birkh{\"a}user, Cham, 2015.
\newblock ISBN 978-3-319-20828-2.
\newblock \doi{10.1007/978-3-319-20828-2}.

\bibitem[Warin(2024)]{Warin}
X.~Warin.
\newblock The groupmax neural network approximation of convex functions.
\newblock \emph{IEEE Transactions on Neural Networks and Learning Systems},
  35\penalty0 (8):\penalty0 11608--11612, 2024.
\newblock \doi{10.1109/TNNLS.2023.3240183}.

\bibitem[Weaver(1999)]{Weaver}
N.~Weaver.
\newblock \emph{Lipschitz Algebras}.
\newblock WORLD SCIENTIFIC, 2nd edition, 1999.
\newblock \doi{10.1142/4100}.
\newblock URL \url{https://www.worldscientific.com/doi/abs/10.1142/4100}.

\end{thebibliography}
\end{document}